\documentclass[oneside,french,american]{amsart}
\usepackage[T1]{fontenc}
\usepackage[latin9]{inputenc}
\usepackage{verbatim}
\usepackage{amsbsy}
\usepackage{amstext}
\usepackage{amsthm}
\usepackage{amssymb}
\usepackage{graphicx}
\PassOptionsToPackage{normalem}{ulem}
\usepackage{ulem}

\makeatletter
\numberwithin{equation}{section}
\numberwithin{figure}{section}
\theoremstyle{plain}
\newtheorem{thm}{\protect\theoremname}[section]
  \theoremstyle{plain}
  \newtheorem{prop}[thm]{\protect\propositionname}
  \theoremstyle{plain}
  \newtheorem{lem}[thm]{\protect\lemmaname}
  \theoremstyle{plain}
  \newtheorem{fact}[thm]{\protect\factname}
  \theoremstyle{plain}
  \newtheorem{cor}[thm]{\protect\corollaryname}

\makeatletter

\usepackage{bbm}
\usepackage{graphicx}
\usepackage{tikz-qtree}

\allowdisplaybreaks

\newtheorem{hypothesis}{Assumption}

\numberwithin{equation}{section}

\newcommand{\1}{\mathbbm{1}}

\newcommand{\N}{\mathbb{N}}
\newcommand{\Z}{\mathbb{Z}}

\newcommand{\E}{\mathbb{E}}
\newcommand{\R}{\mathbb{R}}
\newcommand{\p}{\mathbb{P}}

\renewcommand{\epsilon}{ \varepsilon}

\newcommand{\Bsym}{\mathcal{B}_{\text{sym}}^0}

\newcommand{\Ue}{\mathcal{U}_\epsilon}

\DeclareMathOperator*{\Id}{Id}

\makeatother

\usepackage{a4wide}

\makeatother

\usepackage{babel}
\makeatletter
\addto\extrasfrench{%
   \providecommand{\fg}{\ifdim\lastskip>\z@\unskip\fi~\frqq}%
}

\makeatother
  \addto\captionsamerican{\renewcommand{\corollaryname}{Corollary}}
  \addto\captionsamerican{\renewcommand{\factname}{Fact}}
  \addto\captionsamerican{\renewcommand{\lemmaname}{Lemma}}
  \addto\captionsamerican{\renewcommand{\propositionname}{Proposition}}
  \addto\captionsamerican{\renewcommand{\theoremname}{Theorem}}
  \addto\captionsfrench{\renewcommand{\corollaryname}{Corollaire}}
  \addto\captionsfrench{\renewcommand{\factname}{Fait}}
  \addto\captionsfrench{\renewcommand{\lemmaname}{Lemme}}
  \addto\captionsfrench{\renewcommand{\propositionname}{Proposition}}
  \addto\captionsfrench{\renewcommand{\theoremname}{Théorème}}
  \providecommand{\corollaryname}{Corollary}
  \providecommand{\factname}{Fact}
  \providecommand{\lemmaname}{Lemma}
  \providecommand{\propositionname}{Proposition}
\providecommand{\theoremname}{Theorem}

\begin{document}

\title{Central-limit Theorem for conservative fragmentation chains}

\author{Sylvain Rubenthaler}

\email{rubentha@unice.fr}

\date{\today}
\begin{abstract}
We are interested in a fragmentation process. We observe fragments
frozen when their sizes are less than $\epsilon$ ($\epsilon>0$).
Is is known (\cite{bertoin-martinez-2005}) that the empirical measure
of these fragments converges in law, under some renormalization. In
\cite{hoffmann-krell-2011}, the authors show a bound for the rate
of convergence. Here, we show a central-limit theorem, under some
assumptions. 
\end{abstract}

\maketitle

\section{Introduction}

\subsection{Scientific and economic context}

One of the main goals in the mining industry is to extract blocks
of metallic ore and then separate the metal from the valueless material.
To do so, rock is fragmented into smaller and smaller rocks. This
is carried out in a series of steps, the first one being blasting,
after which the material goes through a sequence of crushers. At each
step, the particles are screened, and if they are smaller than the
diameter of the mesh of a classifying grid, they go to the next crusher.
The process stops when the material has a sufficiently small size
(more precisely, small enough to enable physicochemical processing).

This fragmentation process is energetically costly (each crusher consumes
a certain quantity of energy to crush the material it is fed). One
of the problems that faces the mining industry is that of minimizing
the energy used. The optimisation parameters are the number of crushers
and the technical specifications of these crushers.

In \cite{bertoin-martinez-2005}, the authors propose a mathematical
model of what happens in a crusher. In this model, the rock pieces/fragments
are fragmented independently of each other, in a random and auto-similar
manner. This is consistent with what is observed in the industry,
and this is supported by the following publications: \cite{bird-perrier-2002,devoto-martinez-1998,weiss-1986,turcotte-1986}.
Each fragment has a size $s$ (in $\R^{+}$) and is then fragmented
into smaller fragments of sizes $s_{1}$, $s_{2}$, \ldots{} such
that the sequence $(s_{1}/s,s_{2}/s,\dots)$ has a law $\nu$ which
does not depend on $s$ (which is why the fragmentation is said to
be auto-similar). This law $\nu$ is called the \emph{dislocation
measure} (each crusher has its own dislocation measure). The dynamic
of the fragmentation process is thus modelized in a stochastic way.

In each crusher, the rock pieces are fragmented repetitively until
they are small enough to slide through a mesh whose holes have a fixed
diameter. So the fragmentation process stops for each fragment when
its size is smaller than the diameter of the mesh, which we denote
by $\epsilon$ ($\epsilon>0$). We are interested in the \emph{statistical
distribution }of the fragments coming out of a crusher. If we renormalize
the sizes of these fragments by dividing them by $\epsilon$, we obtain
a measure $\gamma_{-\log(\epsilon)}$, which we call the \emph{empirical
measure }(the reason for the index $-\log(\epsilon)$ instead of $\epsilon$
will be made clear later). In \cite{bertoin-martinez-2005}, the authors
show that the energy consumed by the crusher to reduce the rock pieces
to fragments whose diameters are smaller than $\epsilon$ can be computed
as an integral of a bounded function against the measure $\gamma_{-\log(\epsilon)}$
(they cite \cite{bond-1952,charles-1957,walker-lewis-mcadames-gilliland-1985}
on this particular subject). For each crusher, the empirical measure
$\gamma_{-\log(\epsilon)}$ is one of the two only observable variables
(the other one being the size of the pieces pushed into the grinder).
The specifications of a crusher are summarized in $\epsilon$ and
$\nu$.

\subsection{State of the art}

In \cite{bertoin-martinez-2005}, the authors show that the energy
consumed by a crusher to reduce rock pieces of a fixed size into fragments
whose diameter are smaller than $\epsilon$ behaves asymptotically
like a power of $\epsilon$ when $\epsilon$ goes to zero. More precisely,
this energy multiplied by a power of $\epsilon$ converges towards
a constant of the form $\kappa=\nu(\varphi)$ (the integral of $\nu$,
the dislocation measure, against a bounded function $\varphi$). In
\cite{bertoin-martinez-2005}, the authors also show a law of large
numbers for the empirical measure $\gamma_{-\log(\epsilon)}$. More
precisely, if $f$ is bounded continuous, $\gamma_{-\log(\epsilon)}(f)$
converges in law, when $\epsilon$ goes to zero, towards an integral
of $f$ against a measure related to $\nu$ (this result also appears
in \cite{hoffmann-krell-2011}, p. 399). We set $\gamma_{\infty}(f)$
to be this limit (check Equations (\ref{eq:def-gamma-infty}), (\ref{eq:def-eta}),
(\ref{eq:def-pi}) to get an exact formula). The empirical measure
$\gamma_{-\log(\epsilon)}$ thus contains information relative to
$\nu$ and one could extract from it an estimation of $\kappa$ or
of an integral of any function against $\nu$. 

It is worth noting that by studying what happens in various crushers,
we could study a family $(\nu_{i}(f_{j}))_{i\in I,j\in J}$ (with
an index $i$ for the number of the crusher and the index $j$ for
the $j$-th test function in a well-chosen basis). Using statistical
learning methods, one could from there make a prediction for $\nu(f_{j}$)
for a new crusher for which we know only the mechanical specifications
(shape, power, frequencies of the rotating parts \ldots{}). It would
evidently be interesting to know $\nu$ before even building the crusher.

In \cite{harris-knobloch-kyprianou-2010}, the authors prove a convergence
result for the empirical measure similar to the one in \cite{bertoin-martinez-2005},
the convergence in law being replaced by an almost sure convergence.
In \cite{hoffmann-krell-2011}, the authors give a bound on the rate
of this convergence, in a $L^{2}$ sense, under the assumption that
the fragmentation is conservative. This assumption means there is
no loss of mass due to the formation of dust during the fragmentation
process.\foreignlanguage{french}{}
\begin{figure}[h]
\selectlanguage{french}%
\[
\begin{array}{ccc}
\gamma_{-\log(\epsilon)} & \underset{\underset{\epsilon\rightarrow0}{\longrightarrow}}{\text{(bound on rate)}} & \gamma_{\infty}\\
 &  & \updownarrow\text{relation }\\
\text{energy}\times(\text{power of \ensuremath{\epsilon})} & \underset{\epsilon\rightarrow0}{\sim} & \kappa=\nu(\varphi)
\end{array}
\]

\caption{\foreignlanguage{american}{State of the art.}}
\selectlanguage{american}%
\end{figure}

So we have convergence results (\cite{bertoin-martinez-2005,harris-knobloch-kyprianou-2010})
of an empirical quantity towards constants of interest (a different
constant for each test function $f$). Using some transformations,
these constants could be used to estimate the constant $\kappa$.
Thus it is natural to ask what is the exact rate of convergence in
this estimation, if only to be able to build confidence intervals.
In \cite{hoffmann-krell-2011}, we only have a bound on the rate.

When a sequence of empirical measures converges to some measure, it
is natural to study the fluctuations, which often turn out to be Gaussian.
For such results in the case of empirical measures related to the
mollified Boltzmann equation, one can cite \cite{meleard-1998,uchiyama-1988,dawson-zheng-1991}.
When interested in the limit of a $n$-tuple as in Equation (\ref{eq:n-uplet})
below, we say we are looking at the convergence of a $U$-statistics.
Textbooks deal with the case where the points defining the empirical
measure are independent or with a known correlation (see \cite{delapena-gine-1999,dynkin-mandelbaum-1983,lee-1990}).
The problem is more complex when the points defining the empirical
measure are in interaction with each other like it is the case here.

\subsection{Goal of the paper}

As explained above, we want to obtain the rate of convergence in the
convergence of $\gamma_{-\log(\epsilon)}$ when $\epsilon$ goes to
zero. We want to produce a central-limit theorem of the kind: for
a bounded continuous $f$, $\epsilon^{\beta}(\gamma_{-\log(\epsilon)}(f)-\gamma_{\infty}(f))$
converges towards a non-trivial measure when $\epsilon$ goes to zero
(the limiting measure will in fact be Gaussian), for some exponent
$\beta$. The technics used will allow us to prove the convergence
towards a multivariate Gaussian of a vector of the kind 
\begin{equation}
\epsilon^{\beta}(\gamma_{-\log(\epsilon)}(f_{1})-\gamma_{\infty}(f_{1}),\dots,\gamma_{-\log(\epsilon)}(f_{n})-\gamma_{\infty}(f_{n}))\label{eq:n-uplet}
\end{equation}
for functions $f_{1}$, \ldots{}, $f_{n}$.

More precisely, if by $Z_{1}$, $Z_{2}$, \ldots{}, $Z_{N}$ we denote
the fragments sizes that go out from a crusher (with mesh diameter
equal to $\epsilon$). We would like to show that for a bounded continuous
$f$, 
\[
\gamma_{-\log(\epsilon)}(f):=\sum_{i=1}^{N}Z_{i}f(Z_{i})\longrightarrow\gamma_{\infty}(f)\text{, almost surely, when }\epsilon\rightarrow0\,,
\]
and that for all $n$, and $f_{1}$, \ldots{},$f_{n}$ bounded continuous
function such that $\gamma_{\infty}(f_{i})=0$,
\[
\epsilon^{\beta}(\gamma_{-\log(\epsilon)}(f_{1}),\dots,\gamma_{-\log(\epsilon)}(f_{n}))
\]
converges in law towards a multivariate Gaussian when $\epsilon$
goes to zero. 

The exact results are stated in Proposition \ref{prop:convergence-ps}
and Theorem \ref{thm:central-limit}.

\subsection{Outline of the paper}

We will state our assumptions along the way (Assumptions \ref{hyp:A},
\ref{hyp:conservative}, \ref{hyp:delta-step}, \ref{hyp:queue-pi}).
Assumption \ref{hyp:queue-pi} can be found at the beginning of Section
\ref{sec:Rate-of-convergence}. We define our model in Section \ref{sec:Statistical-model}.
The main idea is that we want to follow tags during the fragmentation
process. Let us imagine the fragmentation is the process of breaking
a stick (modeled by $[0,1]$) into smaller sticks. We suppose that
the original stick has painted dots and that during the fragmentation
process, we take note of the sizes of the sticks supporting the painted
dots (we call them the painted sticks). When the sizes of the painted
sticks get smaller than $\epsilon$ ($\epsilon>0$), the fragmentation
is stopped for these sticks. In Section \ref{sec:Rate-of-convergence},
we make use of classical results on renewal processes and of \cite{sgibnev-2002}
to show that the size of one painted stick has an asymptotic behavior
when $\epsilon$ goes to zero and that we have a bound on the rate
with which it reaches this behavior. Section \ref{sec:Limits-of-symmetric}
is the most technical. There we study the asymptotics of symmetric
functionals of the sizes of the painted sticks (always when $\epsilon$
goes to zero). In Section \ref{sec:Results}, we precisely define
the measure we are interested in ($\gamma_{T}$ with $T=-\log(\epsilon)$).
Using the results of Section \ref{sec:Limits-of-symmetric}, it is
then easy to show a law of large numbers for $\gamma_{T}$ (Proposition
\ref{prop:convergence-ps}) and a central-limit Theorem (Theorem \ref{thm:central-limit}).
Proposition \ref{prop:convergence-ps} and Theorem \ref{thm:central-limit}
are our two main results. The proof of Theorem \ref{thm:central-limit}
is based on a simple computation involving characteristic functions
(the same technique was already used in \cite{del-moral-patras-rubenthaler-2009,del-moral-patras-rubenthaler-2011a,del-moral-patras-rubenthaler-2011b,rubenthaler-2016}).

\subsection{Notations}

For $x$ in $\R$, we set $\lceil x\rceil=\inf\{n\in\Z\,:\,n\geq x\}$,
$\lfloor x\rfloor=\sup\{n\in\Z\,:\,n\leq x\}$. The symbol $\sqcup$
means ``disjoint union''. For $n$ in $\N^{*}$, we set $[n]=\{1,2,\dots,n\}$.
For $f$ an application from a set $E$ to a set $F$, we write $f:E\hookrightarrow F$
if $f$ is injective and, for $k$ in $\N^{*}$, if $F=E$, we set
\[
f^{\circ k}=\underset{k\mbox{ times }}{\underbrace{f\circ f\circ\dots\circ f}}
\]

\section{Statistical model\label{sec:Statistical-model}}

\subsection{Fragmentation chains}

Let $\epsilon>0$. Like in \cite{hoffmann-krell-2011}, we start with
the space 
\[
\mathcal{S}^{\downarrow}=\left\{ \mathbf{s}=(s_{1},s_{2},\dots),\,s_{1}\geq s_{2}\geq\dots\geq0,\,\sum_{i=1}^{+\infty}s_{i}\leq1\right\} \,.
\]
A fragmentation chain is a process in $\mathcal{S}^{\downarrow}$
characterized by 
\begin{itemize}
\item a dislocation measure $\nu$ which is a finite measure on $\mathcal{S}^{\downarrow}$,
\item a description of the law of the times between fragmentations.
\end{itemize}
A fragmentation chain with dislocation measure $\nu$ is a Markov
process $X=(X(t),t\geq0)$ with values in $\mathcal{S}^{\downarrow}$.
Its evolution can be described as follows: a fragment with size $x$
lives for some time (which may or may not be random) then splits and
gives rise to a family of smaller fragments distributed as $x\xi$,
where $\xi$ is distributed according to $\nu(.)/\nu(\mathcal{S}^{\downarrow})$.
We suppose the life-time of a fragment of size $x$ is an exponential
time of parameter $x^{\alpha}\nu(\mathcal{S}^{\downarrow})$, for
some $\alpha$. We could here make different assumptions on the life-time
of fragments, but this would not change our results. %

{} 

We denote by $\p_{m}$ the law of $X$ started from the initial configuration
$(m,0,0,\dots)$ with $m$ in $(0,1]$. The law of $X$ is entirely
determined by $\alpha$ and $\nu(.)$ (Theorem 3 of \cite{bertoin-2002}). 

We make the same assumption as in \cite{hoffmann-krell-2011} and
we will call it Assumption \ref{hyp:A}.

\begin{hypothesis}\label{hyp:A}

We have $\nu(\mathcal{S}^{\downarrow})=1$ and $\nu(s_{1}\in]0;1[)=1$.

\end{hypothesis}

Let
\[
\mathcal{U}:=\{0\}\cup\bigcup_{n=1}^{+\infty}(\N^{*})^{n}
\]
denote the infinite genealogical tree. For $u=(u_{1},\dots,u_{n})\in\mathcal{U}$
and $i\in\N^{*}$, we say that $u$ is in the $n$-th generation and
we write $|u|=n$, and we write $ui=(u_{1},\dots,u_{n},i)$, $u(k)=(u_{1},\dots,u_{k})$
for all $k\in[n]$. For any $u=(u_{1},\dots,u_{n})$ and $v=ui$ ($i\in\N^{*}$),
we say that $u$ is the ancestor of $v$. For any $u$ in $\mathcal{U}\backslash\{0\}$
($\mathcal{U}$ deprived of its root), $u$ has exactly one ancestor
and we denote it by $\mathbf{a}(u)$. The set $\mathcal{U}$ is ordered
alphanumerically~:
\begin{itemize}
\item If $u$ and $v$ are in $\mathcal{U}$ and $|u|<|v|$ then $u<v$.
\item If $u$ and $v$ are in $\mathcal{U}$ and $|u|=|v|=n$ and $u=(u_{1},\dots,u_{n})$,
$v=(v_{1},\dots,v_{n})$ with $u_{1}=v_{1}$, \dots{} , $u_{k}=v_{k}$,
$u_{k+1}<v_{k+1}$ then $u<v$.
\end{itemize}
A mark is an application from $\mathcal{U}$ to some other set. We
associate a mark on the tree $\mathcal{U}$ to each path of the process
$X$. The mark at node $u$ is $\xi_{u}$, where $\xi_{u}$ is the
size of the fragment indexed by $u$.%
{} The distribution of this random mark can be described recursively
as follows.
\begin{prop}
(Consequence \label{prop:(reformulation-of-Proposition} of Proposition
1.3, p. 25, \cite{bertoin-2006}) There exists a family of i.i.d.
variables indexed by the nodes of the genealogical tree, $((\widetilde{\xi}_{ui})_{i\in\N^{*}},u\in\mathcal{U})$,
where each $(\widetilde{\xi}_{ui})_{i\in\N^{*}}$ is distributed according
to the law $\nu(.)/\nu(\mathcal{S}^{\downarrow})$, and such that
the following holds:\\
Given the marks $(\xi_{v},|v|\leq n)$ of the first $n$ generations,
the marks at generation $n+1$ are given by 
\[
\xi_{ui}=\widetilde{\xi}_{ui}\xi_{u}\,,
\]
where $u=(u_{1,}\dots,u_{n})$ and $ui=(u_{1},\dots,u_{n},i)$ is
the $i-th$ child of $u$.
\end{prop}

\subsection{Tagged fragments}

From now on, we suppose that we start with a block of size $m=1$.
We assume that the total mass of the fragments remains constant through
time, as follows.

\begin{hypothesis}\textbf{\label{hyp:conservative}(Conservative
property).}

We have $\nu(\sum_{i=1}^{+\infty}s_{i}=1)=1$.

\end{hypothesis}

\subsubsection{\label{subsec:First-definition}First definition}

We can now define tagged fragments.%
{} We use the representation of fragmentation chains as random infinite
marked tree to define a fragmentation chain with $q$ tagged fragments.
Suppose we have a fragmentation process $X$. On each node $u\in\mathcal{U}$,
we set a mark
\[
(\xi_{u},A_{u})\,,
\]
with $\xi_{u}$ defined as above and $A_{u}\subset[q]$, denoting
the tags present on the fragment labeled by $u$. The random variables
$(A_{u})_{u\in\mathcal{U}}$ are defined as follows.
\begin{itemize}
\item We set $A_{\{0\}}=[q]$.
\item We suppose we have i.i.d. random variables $((U_{u,j})_{j\in[q]},u\in\mathcal{U})$
of law $\mathcal{U}([0,1])$. For all $n\in\N$, given the marks of
the first $n$ generations, the marks at generation $n+1$ are given
by Proposition \ref{prop:(reformulation-of-Proposition} (concerning
$\xi_{.}$) and 
\[
A_{ui}=\{j\in A_{u}\,:\,\widetilde{\xi}_{u1}+\dots+\widetilde{\xi}_{u(i-1)}\leq U_{u,j}<\widetilde{\xi}_{u1}+\dots+\widetilde{\xi}_{u(i-1)}+\widetilde{\xi}_{ui}\}\,,\,\forall u\,:\,|u|=n\,,\,\forall i\in\N^{*}\,.
\]
We observe that, for all $j\in[q]$, $u\in\mathcal{U}$, $i\in\N^{*}$,
\begin{equation}
\p(j\in A_{ui}|j\in A_{u},\widetilde{\xi}_{ui})=\widetilde{\xi}_{ui}\,.\label{eq:proba-rester-dans-fragment}
\end{equation}
\end{itemize}
In the case $q=1$, the branch $\{u\in\mathcal{U}\,:\,A_{u}\neq\emptyset\}$
has the same law as the randomly tagged branch of Section 1.2.3 of
\cite{bertoin-2006}. The presentation is simpler in our case because
the Malthusian exponent is $1$ under Assumption \ref{hyp:conservative}.

\subsubsection{Second definition}

There is a different way to define the law of the random mark $(\xi_{u},A_{u})$,
which we will present now. This definition is strictly equivalent
to the first definition above. We take $(Y_{1},Y_{2},\dots,Y_{q})$
to be $q$ i.i.d. variables of law $\mathcal{U}([0,1])$. We set,
for all $u$ in $\mathcal{U}$,
\[
(\xi_{u},l_{u},A_{u})
\]
with $\xi_{u}$ defined as above. The random variables $A_{u}$ take
values in the subsets of $[q]$. The random variables $l_{u}$ take
values in $[0,1]$. These variables are defined as follows.
\begin{itemize}
\item We set $A_{\{0\}}=[q]$, $l_{\{0\}}=0$.
\item For all $n\in\N$, given the marks of the first $n$ generations,
the marks at generation $n+1$ are given by Proposition \ref{prop:(reformulation-of-Proposition}
(concerning $\xi_{.}$) and 
\[
l_{ui}=l_{u}+\xi_{u}(\widetilde{\xi}_{u1}+\widetilde{\xi}_{u2}+\dots+\widetilde{\xi}_{u(i-1)})\,,\,\forall u\,:\,|u|=n\,,\,\forall i\in\N^{*}\,,
\]
\[
k\in A_{ui}\text{ if and only if }Y_{k}\in[l_{ui},l_{ui}+\xi_{ui})\,,\,\forall u\,:\,|u|=n\,,\,\forall i\in\N^{*}\,.
\]
\end{itemize}
We obtain $(\xi_{u},A_{u})_{u\in\mathcal{U}}$ having the same law
as in Section \ref{subsec:First-definition}. So the two definitions
are equivalent.

\subsection{Observation scheme}

We ofreeze the process when the fragments become smaller than a given
threshold $\epsilon>0$. That is, we have the following data
\[
(\xi_{u})_{u\in\mathcal{U}_{\epsilon}}\,,
\]
where 
\[
\mathcal{U}_{\epsilon}=\{u\in\mathcal{U},\,\xi_{\mathbf{a}(u)}\geq\epsilon,\,\xi_{u}<\epsilon\}\,.
\]

We now look at $q$ tagged fragments ($q\in\N^{*}$). For each $i$
in $[q]$, we call $L_{0}^{(i)}=1$, $L_{1}^{(i)}$, $L_{2}^{(i)}$\dots{}
the successive sizes of the fragment having the tag $i$. More precisely,
for each $n\in\N^{*}$, there is almost surely exactly one $u\in\mathcal{U}$
such that $|u|=n$, $i\in A_{u}$; and so, $L_{n}^{(i)}=\xi_{u}$.
For each $i$, the process $S_{0}^{(i)}=-\log(L_{0}^{(i)})=0\leq S_{1}^{(i)}=-\log(L_{1}^{(i)})\leq\dots$
is a renewal process without delay, with waiting-time following a
law $\pi$ (see \cite{asmussen-2003}, Chapter V for an introduction
to renewal processes). This law $\pi$ is defined by the following.
\begin{equation}
\text{For all bounded measurable }f:[0,1]\rightarrow[0,+\infty)\,,\,\int_{\mathcal{S}^{\downarrow}}\sum_{i=1}^{+\infty}s_{i}f(s_{i})\nu(d\boldsymbol{s})=\int_{0}^{+\infty}f(e^{-x})\pi(dx)\,,\label{eq:def-pi}
\end{equation}
(see Proposition 1.6, p. 34 of \cite{bertoin-2006}, or Equations
(3), (4), p. 398 of \cite{hoffmann-krell-2011}). 

We make the following assumption on $\pi$. 

\begin{hypothesis}\label{hyp:delta-step}

There exist $a,b>0$ ($a<b$) such that the support of $\pi$ is $[a,b]$.
We set $\delta=e^{-b}$. 

\end{hypothesis}

We set 
\[
T=-\log(\epsilon)\,.
\]
 We set, for all $i\in[q]$, $t\geq0$, \textbf{
\begin{equation}
B_{t}^{(i)}=\inf\{S_{j}^{(i)}\,:\,S_{j}^{(i)}>t\}-t\,.\label{eq:def-B}
\end{equation}
}The process $B^{(i)}$ is a homogeneous Markov process (Proposition
1.5 p. 141 of \cite{asmussen-2003}). We call it the residual lifetime
of the fragment tagged by $i$. In the following, we will treat $t$
as a time parameter. This has nothing to do with the time in which
the fragmentation process $X$ evolves. 

We observe that, for all $t$, $(B_{t}^{(1)},\dots,B_{t}^{(q)})$
is exchangeable (meaning that for all $\sigma$ in the symmetric group
of order $q$, $(B_{t}^{(\sigma(1))},\dots,B_{t}^{(\sigma(q))})$
has the same law as $(B_{t}^{(1)},\dots,B_{t}^{(q)})$).

\subsection{Stationary age process\label{subsec:Stationary-age-process}}

We define $\widetilde{X}$ to be an independent copy of $X$. We suppose
it has $q$ tagged fragments. Therefore it has a mark $(\widetilde{\xi},\widetilde{A})$
and renewal processes $(\widetilde{S}_{k}^{(i)})_{k\geq0}$ (for all
$i$ in $[q]$) defined in the same way as for $X$. We let $(\widetilde{B}^{(1)},\widetilde{B}^{(2)})$
be the residual lifetimes of the fragments tagged by $1$ and $2$.

Let 
\[
\mu=\int_{0}^{+\infty}x\pi(dx)
\]
 and let $\pi_{1}$ be the distribution with density $x\mapsto x/\mu$
with respect to $\pi$. We set $\overline{C}$ to be a random variable
of law $\pi_{1}$. We set $U$ to be independent of $\overline{C}$
and uniform on $(0,1)$. We set $\widetilde{S}_{-1}=\overline{C}(1-U)$.
The process $\overline{S}_{0}=\widetilde{S}_{-1}$, $\overline{S}_{1}=\widetilde{S}_{-1}+\widetilde{S}_{0}^{(1)}$
, $\overline{S}_{2}=\widetilde{S}_{-1}+\widetilde{S}_{1}^{(1)}$,
$\overline{S}_{2}=\widetilde{S}_{-1}+\widetilde{S}_{2}^{(1)}$, \ldots{}
is a renewal process with delay $\pi_{1}$. We set $(\overline{B}_{t}^{(1)})_{t\geq0}$
to be its residual lifetime process~:
\begin{equation}
\overline{B}_{t}^{(1)}=\begin{cases}
\overline{C}(1-U)-t & \mbox{ if }t<\overline{S}_{0}\,,\\
\inf_{n\geq0}\{\overline{S}_{n}\,:\,\overline{S}_{n}>t\}-t & \mbox{ if }t\geq\overline{S}_{0}\,.
\end{cases}\label{eq:def-B-stationnaire}
\end{equation}
Theorem 3.3 p.151 of \cite{asmussen-2003} tells us that $(\overline{B}_{t}^{(1)})_{t\geq0}$
has the same transition as $(B_{t}^{(1)})_{t\geq0}$ defined above
and that $(\overline{B}_{t}^{(1)})_{t\geq0}$ is stationary. 

We define a measure $\eta$ on $\R^{+}$ by its action on bounded
measurable functions:
\begin{equation}
\text{For all bounded measurable }f\,:\,\R^{+}\rightarrow\R\,,\,\eta(f)=\frac{1}{\mu}\int_{\R^{+}}\E(f(Y-s)\1_{\{Y-s\geq0\}})ds\,,\,(Y\sim\pi)\,.\label{eq:def-eta}
\end{equation}
\begin{lem}
The measure $\eta$ is the law of $\overline{B}_{t}^{(1)}$ (for any
$t$).
\end{lem}

\begin{proof}
Let $\xi\geq0$. We set $f(y)=\1_{y\geq\xi}$, for all $y$ in $\R$.
We have (with $Y$ of law $\pi$)
\begin{eqnarray*}
\frac{1}{\mu}\int_{\R^{+}}\E(f(Y-s)\1_{Y-s\geq0})ds & = & \frac{1}{\mu}\int_{\R^{+}}\left(\int_{0}^{y}\1_{y-s\geq\xi}ds\right)\pi(dy)\\
 & = & \frac{1}{\mu}\int_{\R^{+}}(y-\xi)_{+}\pi(dy)\\
 & = & \int_{\xi}^{+\infty}\left(1-\frac{\xi}{y}\right)\frac{y}{\mu}\pi(dy)\\
 & = & \p(\overline{C}(1-U)\geq\xi)\,.
\end{eqnarray*}
\end{proof}
For $v$ in $\R$, we now want to define a process $(\overline{B}_{t}^{(1),v})_{t\geq v}$
having the same transition as $B_{t}^{(1)}$ and being stationary.
We set $\overline{B}_{v}^{(1),v}$ such that it has the law $\eta$.
As we have given its transition, the process $(\overline{B}_{t}^{(1),v})_{t\geq v}$
is well defined in law. In addition, we suppose that it is independent
of all the other processes. 

For $v$ in $[0,T]$, we define a process $(\widehat{B}^{(1),v},\widehat{B}^{(2),v})$
such that $\widehat{B}^{(1),v}=B^{(1)}$ and $(\widehat{B}^{(1),v},\widehat{B}^{(2),v})$
has the law of $(B^{(1)},B^{(2)})$ conditioned on 
\[
\forall u\in\mathcal{U}\,,\,1\in A_{u}\Rightarrow[2\in A_{u}\Leftrightarrow-\log(\xi_{u})\leq v]\,,
\]
which reads as follows~: the tag $2$ remains on the fragment bearing
the tag $1$ until the size of the fragment is smaller than $e^{-v}$.
We observe that, conditionally on $\widehat{B}_{v}^{(1),v}$, $\widehat{B}_{v}^{(2),v}$:
$(\widehat{B}_{v+\widehat{B}_{v}^{(1),v}+t}^{(1),v})_{t\geq0}$ and
$(\widehat{B}_{v+\widehat{B}_{v}^{(2),v}+t}^{(2),v})_{t\geq0}$ are
independent. 

Let $k$ in $\N^{*}$ be such that 
\begin{equation}
(k-1)\times(b-a)\geq a\,.\label{eq:number-of-steps}
\end{equation}
 Now we state a small Lemma that will be useful below.
\begin{lem}
Let $v$ be in $\R.$ The variables $\overline{B}_{v}^{(1),v}$ and
$\widehat{B}_{kb}^{(1),kb}$ have the same support (and it is $[0,-\log(\delta)]$).
\end{lem}

\begin{proof}
{} %
{} By Equation (\ref{eq:def-B-stationnaire}), the support of $\eta$
is $[0,b]$; and so the support of $\overline{B}_{v}^{(1),v}$ is
$[0,b]$. By Assumption \ref{hyp:delta-step}, the support of $S_{k-1}^{(1)}$
is $[(k-1)a,(k-1)b]$ and the support of $S_{k}^{(1)}-S_{k-1}^{(1)}$
is $[a,b]$. If $S_{k}^{(1)}>(k-1)b$ then $B_{kb}^{(1)}=S_{k}^{(1)}-S_{k-1}^{(1)}-((k-1)b-S_{k-1}^{(1)})$.
As $S_{k-1}^{(1)}$ and $S_{k}^{(1)}-S_{k-1}^{(1)}$ are independent,
we  get that the support of $B_{kb}^{(1)}$ includes $[0,b]$ (because
of Equation (\ref{eq:number-of-steps})). As this support is included
in $[0,b]$, we have proved the desired result.
\end{proof}
For $v$ in $\R$, we define a process $(\overline{B}_{t}^{(2),v})_{t\geq v}$
by: $(\overline{B}_{t}^{(1),v},\overline{B}_{t}^{(2),v})_{t\geq v}$
has the law of 
\[
(\widehat{B}_{t-v+kb}^{(1),kb},\widehat{B}_{t-v+kb}^{(2),kb})_{t\geq v}
\]
 conditioned on $(\widehat{B}_{t-v+kb}^{(1),kb})_{t\geq v}=(\overline{B}_{t}^{(1),v})_{t\ge v}$.
This conditioning is correct because $\overline{B}_{v}^{(1),v}$ and
$\widehat{B}_{kb}^{(1),kb}$ have the same support. 

\section{\label{sec:Rate-of-convergence}Rate of convergence in the Key Renewal
Theorem}

We need the following regularity assumption. 

\begin{hypothesis}\label{hyp:queue-pi}

The probability $\pi(dx)$ is absolutely continuous with respect to
the Lebesgue measure (we will write $\pi(dx)=\pi(x)dx$). The density
function $x\mapsto\pi(x)$ is continuous on $(0;+\infty)$. 

\end{hypothesis}
\begin{fact}
Let $\theta>1$ ($\theta$ is fixed in the rest of the paper). The
density $\pi$ satisfies 
\[
\limsup_{x\rightarrow+\infty}\exp(\theta x)\pi(x)<+\infty\,.
\]
\end{fact}

For $\varphi$ a nonnegative Borel-measurable function on $\R$, we
set $S(\varphi)$ to be the set of complex-valued measures $\kappa$
(on the Borelian sets) such that $\int_{\R}\varphi(x)|\kappa|(dx)<\infty$,
where $|\kappa|$ stands for the total variation norm. If $\kappa$
is a finite complex-valued measure on the Borelian sets of $\R$,
we define $T\kappa$ to be the $\sigma$-finite measure with the density
\[
v(x)=\begin{cases}
\kappa((x,+\infty)) & \text{ if }x\geq0\,,\\
-\kappa((-\infty,x]) & \text{ if }x<0\,.
\end{cases}
\]
Let $F$ be the cumulative distribution function of $\pi$. %

We set $B_{t}=B_{t}^{(1)}$ (see Equation (\ref{eq:def-B}) for the
definition of $B^{(1)}$, $B^{(2)}$, \ldots{}). By Theorem 3.3 p.151
and Theorem 4.3 p. 156 of \cite{asmussen-2003}, we know that $B_{t}$
converges in law to a random variable $B_{\infty}$(of law $\eta$).
The following Theorem is a consequence of \cite{sgibnev-2002}, Theorem
5.1, p. 2429. It shows there is actually a rate of convergence for
this convergence in law. 
\begin{thm}
\label{lem:sgibnev}Let $\epsilon'\in(0,\theta)$ . Let 
\[
\varphi(x)=\begin{cases}
e^{(\theta-\epsilon')x} & \text{ if }x\geq0\,,\\
1 & \text{ if }x<0\,.
\end{cases}
\]
 %
If $Y$ is a random variable of law $\pi$ then
\[
\sup_{\alpha\,:\,|\alpha|\leq M}\left|\E(\alpha(B_{t}))-\frac{1}{\mu}\int_{\R^{+}}\E(\alpha(Y-s)\1_{\{Y-s\geq0\}})ds\right|=o\left(\frac{1}{\varphi(t)}\right)
\]
as $t$ approaches $+\infty$ outside a set of Lebesgue measure zero
(the supremum is taken on $\alpha$ in the set of Borel-measurable
functions on $\R$). 
\end{thm}

\begin{proof}
{} Let $*$ stands for the convolution product. We define the renewal
measure $U(dx)=\sum_{n=0}^{+\infty}\pi^{*n}(dx)$ (notations: $\pi^{*0}(dx)=\delta_{0}$,
the Dirac mass at $0$, $\pi^{*n}=\pi*\pi*\dots*\pi$ ($n$ times)).
 We take i.i.d. variables $X,X_{1},X_{2}\dots$ of law $\pi$. We
set $f(x)=M$, for all $x$ in $\R$. We have, for all $t\geq0$,
\begin{eqnarray*}
\E(f(B_{t})) & = & \E\left(\sum_{n=0}^{+\infty}f(X_{1}+X_{2}+\dots+X_{n+1}-t)\1_{\{X_{1}+\dots+X_{n}<t\leq X_{1}+\dots+X_{n+1}\}}\right)\\
 & = & \int_{0}^{t}\E(f(s+X-t)\1_{\{s+X-t\geq0\}})U(ds)\,.
\end{eqnarray*}
We set 
\[
g(t)=\begin{cases}
\E(f(X-t)\1_{\{X-t\geq0\}}) & \text{ if }t\geq0\,,\\
0 & \text{ if }t<0\,.
\end{cases}
\]
{} We have, for all $t\geq0$,
\begin{eqnarray*}
\left|\E(f(X-t)\1_{\{X-t\geq0\}})\right| & \leq & \Vert f\Vert_{\infty}\p(X\geq t)\\
 & \leq & \Vert f\Vert_{\infty}e^{-(\theta-\frac{\epsilon'}{2})t}\E(e^{(\theta-\frac{\epsilon'}{2})X})\,.
\end{eqnarray*}
We have: $\E(e^{(\theta-\frac{\epsilon'}{2})X})<\infty$. The function
$\varphi$ is submultiplicative and it is such that 
\[
\lim_{x\rightarrow-\infty}\frac{\log(\varphi(x))}{x}=0\leq\lim_{x\rightarrow+\infty}\frac{\log(\varphi(x))}{x}=\theta-\epsilon'\,.
\]
 The function $g$ is in $L^{1}(\R)$. The function $g.\varphi$ is
in $L^{\infty}(\R)$. We have $g(x)\varphi(x)\rightarrow0$ as $|x|\rightarrow\infty$.
We have
\[
\varphi(t)\int_{t}^{+\infty}|g(x)|dx\underset{t\rightarrow+\infty}{\longrightarrow}0\,,\,\varphi(t)\int_{-\infty}^{t}|g(x)|dx\underset{t\rightarrow-\infty}{\longrightarrow}0\,.
\]
We have $T^{2}(\pi)\in S(\varphi)$. 

Let us now take  a function $\alpha$ such that $|\alpha|\leq M$.
We set
\[
\widehat{\alpha}(t)=\begin{cases}
\E(\alpha(X-t)\1_{\{X-t\geq0\}}) & \text{ if }t\geq0\,,\\
0 & \text{ if }t<0\,.
\end{cases}
\]
Then we have $|\widehat{\alpha}|\leq|g|$ and (computing as above
for $f$)
\begin{eqnarray*}
\E(\alpha(B_{t})) & = & \widehat{\alpha}*U(t)
\end{eqnarray*}

So, by \cite{sgibnev-2002}, Theorem 5.1, we have proved the desired
result. 
\end{proof}
\begin{cor}
\label{cor:sgibnev}There exists a constant $\Gamma_{1}$ bigger than
$1$ such that: for any bounded measurable function $F$ on $\R$
such that $\eta(F)=0$,
\[
|\E(F(B_{t}))|\leq\Vert F\Vert_{\infty}\times\frac{\Gamma_{1}}{\varphi(t)}
\]
for $t$ outside a set of Lebesgue measure zero.
\end{cor}

\begin{proof}
We take $M=1$ in the above Theorem. Keep in mind  that $\eta$ is
defined in Equation (\ref{eq:def-eta}). There exists a constant $\Gamma_{1}$
such that: for all measurable function $\alpha$ such that $\Vert\alpha\Vert_{\infty}\leq1$,
\begin{equation}
\left|\E(\alpha(B_{t}))-\eta(\alpha)\right|\leq\frac{\Gamma_{1}}{\varphi(t)}\,\text{(for }t\text{ outside a set of Lebesque measure zero).}\label{eq:rate-01}
\end{equation}
Let us now take a bounded measurable $F$ such that $\eta(F)=0$.
By Equation (\ref{eq:rate-01}), we have (for $t$ outside a set of
Lebesgue measure zero)
\begin{eqnarray*}
\left|\E\left(\frac{F(B_{t})}{\Vert F\Vert_{\infty}}\right)-\eta\left(\frac{F}{\Vert F\Vert_{\infty}}\right)\right| & \leq & \frac{\Gamma_{1}}{\varphi(t)}\\
|\E(F(B_{t}))| & \leq & \Vert F\Vert_{\infty}\times\frac{\Gamma_{1}}{\varphi(t)}\,.
\end{eqnarray*}
\end{proof}

\section{\label{sec:Limits-of-symmetric}Limits of symmetric functionals }

\subsection{Notations}

We fix $q\in\N^{*}$. We set $\mathcal{S}_{q}$ to be the symmetric
group of order $q$. A function $F:\R^{q}\rightarrow\R$ is symmetric
if
\[
\forall\sigma\in\mathcal{S}_{q}\,,\,\forall(x_{1},\dots,x_{q})\in\R^{q}\,,\,F(x_{\sigma(1)},x_{\sigma(2)},\dots,x_{\sigma(q)})=F(x_{1},x_{2},\dots,x_{q})\,.
\]
For $F:\R^{q}\rightarrow\R$, we define a symmetric version of $F$
by
\[
F_{\text{sym}}(x_{1},\dots,x_{q})=\frac{1}{q!}\sum_{\sigma\in\mathcal{S}_{q}}F(x_{\sigma(1)},\dots,x_{\sigma(q)})\,,\,\text{for all }(x_{1},\dots,x_{q})\in\R^{q}\,.
\]
We set $\mathcal{B}_{\text{sym }}(q)$ to be the set of bounded, measurable,
symmetric functions $F$ on $\R^{q}$, and we set $\mathcal{B}_{\text{sym }}^{0}(q)$
to be the $F$ of $\mathcal{B}_{\text{sym }}(q)$ such that 
\[
\int_{x_{1}}F(x_{1},x_{2},\dots,x_{q})\eta(dx_{1})=0\,,\,\forall(x_{2},\dots,x_{q})\in\R^{q-1}\,.
\]
We set
\[
L_{T}=\sum_{u\in\mathcal{U}_{\epsilon}:A_{u}\neq\emptyset}(\#A_{u}-1)\,.
\]
Suppose that $k$ is in $[q]$ and $l\geq1$. %
For $t$ in $[0,T]$, we consider the following collections of nodes
of $\mathcal{U}$~:
\[
\mathcal{T}_{1}=\{u\in\mathcal{U}\backslash\{0\}\,:\,A_{u}\neq\emptyset\,,\,\xi_{\mathbf{a}(u)}\geq\epsilon\}\cup\{0\}\,,
\]
\[
S(t)=\{u\in\mathcal{T}_{1}\,:\,-\log(\xi_{\mathbf{a}(u)})\leq t\,,\,-\log(\xi_{u})>t\}\,.
\]
We set $\mathcal{L}_{1}$ to be the set of leaves in the tree $\mathcal{T}_{1}$.
For $t$ in $[0,T]$ and $i$ in $[q]$, there exists one and only
one $u$ in $S(t)$ such that $i\in A_{u}$. We call it $u\{t,i\}$.
Under Assumption \ref{hyp:delta-step}, there exists a constant bounding
the numbers vertices of $\mathcal{T}_{1}$ almost surely. Let us look
at an example in Figure \ref{fig:Tree-and-marks}. 
\begin{figure}[h]
\begin{centering}
\includegraphics[scale=0.75]{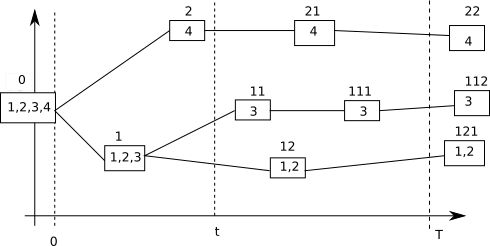}
\par\end{centering}
\caption{\label{fig:Tree-and-marks}Tree and marks}
\end{figure}
Here, we have a graphic representation of a realization of $\mathcal{T}_{1}$.
Each node $u$ of $\mathcal{T}_{1}$ is written above a rectangular
box in which we read $A_{u}$; the right side of the box has the coordinate
$-\log(\xi_{u})$ on the $X$-axis. For simplicity, the node $(1,1)$
is designated by $11$, the node $(1,2)$ is designated by $12$,
and so on. In this example: $\mathcal{T}_{1}=\{0,(1),(2),(1,1),(2,1),(1,2),(1,1,1),(2,2),(1,1,2),(1,2,1)\}$,
$\mathcal{L}_{1}=\{(2,2),(1,1,2),(1,2,1)\}$, $A_{(1)}=\{1,2,3\}$,
$A_{(1,2)}=\{1,2\}$, \dots , $S(t)=\{(1,2),(1,1),(2,1)\}$, $u\{t,1\}=(1,2)$,
$u\{t,2\}=(1,2)$, $u\{t,3\}=(1,1)$, $u\{t,4\}=(2,1)$.

For $k$, $l$ in $\N$, we define the event 
\[
C_{k,l}(t)=\{\sum_{u\in S(t)}\1_{\#A_{u}=1}=k\,,\,\sum_{u\in S(t)}(\#A_{u}-1)=l\}\,.
\]
For example, in Figure \ref{fig:Tree-and-marks}, we are in the event
$C_{2,1}(t)$. 

We define
\[
\mathcal{T}_{2}=\{u\in\mathcal{T}_{1}\backslash\{0\}\,:\,\#A_{\mathbf{a}(u)}\geq2\}\cup\{0\}\,,
\]
\[
m_{2}\,:\,u\in\mathcal{T}_{2}\mapsto(\xi_{u},\inf\{i,i\in A_{u}\})\,.
\]
For example, in Figure \ref{fig:Tree-and-marks}, $\mathcal{T}_{2}=\{(0),(1),(2),(1,1),(1,2),(1,2,1)\}$.
Let $\alpha$ be in $(0,1)$. We observe that $C_{k,l}(\alpha T)$
is measurable with respect to $(\mathcal{T}_{2},m_{2})$ if $T-\alpha T>b$
(we suppose that this is the case in the following). We set, for all
$u$ in $\mathcal{T}_{2}$, $T_{u}=-\log(\xi_{u})$. %
{} Let $\mathcal{L}_{2}$ be the set of leaves $u$ in the tree $\mathcal{T}_{2}$
such that the set $A_{u}$ has a single element $n_{u}$. For example,
in Figure \ref{fig:Tree-and-marks}, $\mathcal{L=}\{(2),(1,1)\}$.

For $q$ even ($q=2p$) and for all $t$ in $[0,T]$, we define the
events
\[
P_{t}=\{\forall i\in[p]\,,\,\exists u_{i}\in\mathcal{U}\,:\,\xi_{u_{i}}<e^{-t}\,,\,\xi_{\mathbf{a}(u_{i})}\geq e^{-t}\,,\,A_{u_{i}}=\{2i-1,2i\}\}\,,
\]

\[
\forall i\in[p]\,,\,P_{i,i+1}(t)=\{\exists u\in S(t)\,:\,\{2i-1,2i\}\subset A_{u}\}\,.
\]
We set, for all $t$ in $[0,T]$, 
\[
\mathcal{F}_{S(t)}=\sigma(S(t),(\xi_{u},A_{u})_{u\in S(t)})\,.
\]

\subsection{Intermediate results }
\begin{lem}
\label{lem:conv-nouille-libre}We suppose that $F$ is in $\Bsym(q)$
and that $F$ is of the form $F=(f_{1}\otimes f_{2}\otimes\dots\otimes f_{q})_{sym}$.
Let $A$ be in $\sigma(\mathcal{L}_{2})$. For any $\alpha$ in $]0,1[$,
$k$ in $[q]$ and $l$ in $\{0,1,\dots,(q-k-1)_{+}\}$, we have
\[
|\E(\1_{C_{k,l}(\alpha T)}\1_{A}F(B_{T}^{(1)},B_{T}^{(2)},\dots,B_{T}^{(q)}))|\leq\Vert F\Vert_{\infty}\Gamma_{1}^{q}C_{tree}(q)\left(\frac{1}{\delta}\right)^{q}\epsilon^{q/2}\,,
\]
(for a constant $C_{tree}(q)$ defined below in the proof) and
\[
\]
\[
\epsilon^{-q/2}\E(\1_{C_{k,l}(\alpha T)}\1_{A}F(B_{T}^{(1)},B_{T}^{(2)},\dots,B_{T}^{(q)}))\underset{\epsilon\rightarrow0}{\longrightarrow}0\,.
\]
\end{lem}

\begin{proof}
We have $\{\#\mathcal{L}_{1}=q\}\in\sigma(\mathcal{L}_{2})$. Let
$A$ be in $\sigma(\mathcal{L}_{2})$. Since the event $C_{k,l}(\alpha T)$
is in $\sigma(\mathcal{L}_{2})\vee\sigma(\mathcal{T}_{2})\vee\sigma(m_{2})$,
we have

\begin{multline*}
|\E(\1_{C_{k,l}(\alpha T)}\1_{A}F(B_{T}^{(1)},B_{T}^{(2)},\dots,B_{T}^{(q)}))|\\
=|\E(\1_{C_{k,l}(\alpha T)}\1_{A}\E(F(B_{T}^{(1)},B_{T}^{(2)},\dots,B_{T}^{(q)})|\mathcal{L}_{2},\mathcal{T}_{2},m_{2}))\\
=|\E(\sum_{f:\mathcal{T}_{2}\rightarrow\mathcal{P}([q])}\1_{C_{k,l}(\alpha T)}\1_{A}\E(F(B_{T}^{(1)},B_{T}^{(2)},\dots,B_{T}^{(q)})\1_{A_{u}=f(u),\forall u\in\mathcal{T}_{2}}|\mathcal{L}_{2},\mathcal{T}_{2},m_{2}))\,.
\end{multline*}

If $u$ in $\mathcal{L}_{2}$ and if $T_{u}<T$, then, conditionally
on $\mathcal{T}_{2}$, $m_{2}$, $B_{T}^{(n_{u})}$ is independent
of all the other variables and has the same law as $B_{T-T_{u}}^{(1)}$.
Thus, using Theorem \ref{lem:sgibnev} and Corollary \ref{cor:sgibnev},
we get, for any $\epsilon'\in(0,\theta-1)$, $u\in\mathcal{L}_{2}$,
$i\in A_{u}$,
\begin{multline*}
\E(f_{i}(B_{T}^{(i)})|\mathcal{L}_{2},\mathcal{T}_{2},m_{2})\leq e^{-(\theta-\epsilon')(T-T_{u})_{+}}\,,\,\\
\text{for }T-T_{u}\notin Z_{0}\text{ where \ensuremath{Z_{0}} is of Lebesgue measure zero.}
\end{multline*}
Thus we get
\begin{multline*}
|\E(\1_{C_{k,l}(\alpha T)}\1_{A}F(B_{T}^{(1)},B_{T}^{(2)},\dots,B_{T}^{(q)}))|\\
\text{(since }F\text{ is of the form }F=(f_{1}\otimes\dots\otimes f_{q})_{sym}\text{,}\\
\text{since, conditionally on }u\in\mathcal{L}_{2}\text{, }\text{the distribution of }T_{u}\text{ is absolutely continuous }\\
\text{with respect to the Lebesgue measure)}\\
\text{ }\\
\leq\Vert F\Vert_{\infty}\Gamma_{1}^{q}\E(\sum_{f:\mathcal{T}_{2}\rightarrow\mathcal{P}([q])}\left[\1_{C_{k,l}(\alpha T)}\prod_{u\in\mathcal{L}_{2}}e^{-(\theta-\epsilon')(T-T_{u})_{+}}\times\1_{A}\E(\1_{A_{u}=f(u),\forall u\in\mathcal{T}_{2}}|\mathcal{L}_{2},\mathcal{T}_{2},m_{2})\right])\\
\text{(because of Assumption \ref{hyp:delta-step} and because }\theta-\epsilon'>1\text{)}\\
\leq\Vert F\Vert_{\infty}\Gamma_{1}^{q}\E(\sum_{f:\mathcal{T}_{2}\rightarrow\mathcal{P}([q])}\left[\1_{C_{k,l}(\alpha T)}\prod_{u\in\mathcal{L}_{2}}e^{-(T-T_{\mathbf{a}(u)})-\log(\delta)}\times\1_{A}\E(\1_{A_{u}=f(u),\forall u\in\mathcal{T}_{2}}|\mathcal{L}_{2},\mathcal{T}_{2},m_{2})\right])\\
\text{(because of Equation (\ref{eq:proba-rester-dans-fragment}))}\\
\leq\Vert F\Vert_{\infty}\Gamma_{1}^{q}\E(\sum_{f:\mathcal{T}_{2}\rightarrow\mathcal{P}([q])}\1_{C_{k,l}(\alpha T)}\1_{A}\left[\prod_{u\in\mathcal{L}_{2}}e^{-(T-T_{\mathbf{a}(u)})-\log(\delta)}\times\prod_{u\in\mathcal{T}_{2}\backslash\{0\}}e^{-(A_{u}-1)(T_{u}-T_{\mathbf{a}(u)})}\right])\,.
\end{multline*}
For a fixed $\omega$, we have
\[
\prod_{u\in\mathcal{L}_{2}}e^{-(T-T_{\mathbf{a}(u)})-\log(\delta)}\times\prod_{u\in\mathcal{T}_{2}\backslash\{0\}}e^{-(A_{u}-1)(T_{u}-T_{\mathbf{a}(u)})}=\left(\frac{1}{\delta}\right)^{\#\mathcal{L}_{2}}\exp\left(-\int_{0}^{T}a(s)ds\right)\,,
\]
where, for all $s$, 
\begin{eqnarray*}
a(s) & = & \sum_{u\in\mathcal{T}_{2}\backslash\{0\}\,:\,T_{a(u)}\leq s<T}\1_{A_{u}=1}+\sum_{u\in\mathcal{T}_{2}\backslash\{0\}\,:\,T_{\mathbf{a}(u)}\leq s\leq T_{u}}(\#A_{u}-1)\\
 & = & \sum_{u\in S(s)}\1_{\#A_{u}=1}+\sum_{u\in S(s)}(\#A_{u}-1)\,.
\end{eqnarray*}
We observe that, for all $\omega$: 
\[
a(t)\geq\left\lceil \frac{q}{2}\right\rceil \,,\,\forall t\,,
\]
\[
a(\alpha T)=k+l\mbox{ , if }\omega\in C_{k,l}(\alpha T)\,,
\]
and if $t$ is such that 
\[
\sum_{u\in S(t)}(\#A_{u}-1)=l'\,,\,\sum_{u\in S(t)}\1_{\#A_{u}=1}=k'
\]
 for some integers $l'$, $k'$, then for all $s\geq t$,
\[
a(s)\geq k'+\left\lceil \frac{q-k'}{2}\right\rceil \,.
\]
We observe that, under Assumption \ref{hyp:delta-step}, there exists
a constant which bounds $\#\mathcal{T}_{1}$ almost surely and so
there exists a constant $C_{tree}(q)$ which bounds $\#\{f:\mathcal{T}_{1}\rightarrow\mathcal{P}([q])\}$
almost surely. So, we have 
\begin{multline*}
|\E(\1_{C_{k,l}(\alpha T)}\1_{A}F(B_{T}^{(1)},B_{T}^{(2)},\dots,B_{T}^{(q)}))|\\
\leq\Vert F\Vert_{\infty}\Gamma_{1}^{q}\E(\sum_{f:\mathcal{T}_{2}\rightarrow\mathcal{P}([q])}\1_{A}\1_{C_{k,l}(\alpha T)}\left(\frac{1}{\delta}\right)^{\#\mathcal{L}_{2}}e^{-\lceil q/2\rceil\alpha T}e^{-(k+\left\lceil \frac{q-k}{2}\right\rceil )(T-\alpha T)})\\
\leq\Vert F\Vert_{\infty}\Gamma_{1}^{q}C_{tree}(q)\left(\frac{1}{\delta}\right)^{q}e^{-\lceil q/2\rceil\alpha T}e^{-(k+\left\lceil \frac{q-k}{2}\right\rceil )(1-\alpha)T}\,.
\end{multline*}
As $k\geq1$, then $k+\left\lceil \frac{q-k}{2}\right\rceil >\frac{q}{2}$,
and so we have proved the desired result. 
\end{proof}
\begin{lem}
\label{lem:conv-nouilles-liees}Let $k$ be an integer $\geq q/2$.
Let $\alpha\in[q/(2k),1]$. We have
\[
\p(L_{\alpha T}\geq k)\leq K_{1}(q)\epsilon^{q/2}\,,
\]
where $K_{1}(q)=\sum_{i\in[q]}\frac{q!}{(q-i)!}$. 

Let $k$ be an integer $>q/2$. Let $\alpha\in(q/(2k),1)$. We have

\[
\epsilon^{-q/2}\p(L_{\alpha T}\geq k)\underset{\epsilon\rightarrow0}{\longrightarrow}0\,.
\]
 
\begin{proof}
Let $k$ be an integer $\geq q/2$ and let $\alpha\in[q/(2k),1]$.
We decompose
\[
\{L_{\alpha T}\geq k\}=\cup_{i\in[q]}\cup_{m:[i]\hookrightarrow[q]}(F(i,m)\cap\{L_{\alpha T}\geq k)\cap\{\#S(\alpha T)=i\})\,,
\]
where
\[
F(i,m)=\{i_{1},i_{2}\in[i]\text{ with }i_{1}\neq i_{2}\Rightarrow\exists u_{1},u_{2}\in S(\alpha T),u_{1}\ne u_{2},m(i_{1})\in A_{u_{1}}\,,\,m(i_{2})\in A_{u_{2}}\}\,.
\]
Suppose we are in the event $F(i,m)$. For $u\in S(\alpha T)$ and
for all $j$ in $[i]$ such that $m(j)\in A_{u}$, we define 
\[
T_{|u|}^{(j)}=-\log(\xi_{u})\,,\,T_{|u|-1}^{(j)}=-\log(\xi_{\mathbf{a}(u)})\,,\,\dots\,,\,T_{1}^{(j)}=-\log(\xi_{\mathbf{a}^{\circ(|u|-1)}(u)})\,,T_{0}^{(j)}=0\,,
\]
\[
l(j)=|u|\,,\,v(j)=u\,.
\]
We have
\begin{multline*}
\p(L_{\alpha T}\geq k)\leq\sum_{i\in[q]}\sum_{m:[i]\hookrightarrow[q]}\p(F(i,m)\cap\{L_{\alpha T}\geq k\}\cap\{\#S(\alpha T)=i\})\\
=\sum_{i\in[q]}\sum_{m:[i]\hookrightarrow[q]}\E(\1_{L_{\alpha T}\geq k}\1_{F(i,m)}\E(\1_{\#S(\alpha T)=i}|F(i,m),L_{\alpha T},(T_{p}^{(j)})_{j\in[i],p\in[l(j)]},\,(v(j))_{j\in[i]},\,(A_{v(j)})_{j\in[i]}))\\
\mbox{(because of Equation (\ref{eq:proba-rester-dans-fragment}))}\\
=\sum_{i\in[q]}\sum_{m:[i]\hookrightarrow[q]}\E\left(\1_{L_{\alpha T}\geq k}\1_{F(i,m)}\prod_{j\in[i]}\prod_{r\in A_{v(j)}\backslash m(j)}\prod_{k=1}^{l(j)}\exp((-T_{k}^{(j)}+T_{k-1}^{(j)}))\right)\\
\leq\sum_{i\in[q]}\sum_{m:[i]\hookrightarrow[q]}\E\left(\1_{L_{\alpha T}\geq k}\1_{F(i,m)}\prod_{j\in[i]}(e^{-\alpha T})^{\#A_{v(j)}-1}\right)\\
\leq\sum_{i\in[q]}\sum_{m:[i]\hookrightarrow[q]}\E(\1_{L_{\alpha T}\geq k}e^{-k\alpha T})\leq e^{-k\alpha T}\times\sum_{i\in[q]}\frac{q!}{(q-i)!}\,.
\end{multline*}
If we suppose that $k>q/2$ and $\alpha\in(q/(2k),1)$, then
\[
\exp\left(\frac{qT}{2}\right)\exp(-k\alpha T)\underset{T\rightarrow+\infty}{\longrightarrow}0\,.
\]
\end{proof}
\end{lem}

Immediate consequences of the two above lemmas are the following Corollaries.
\begin{cor}
\label{cor:lim-case-q-odd}If $q$ is odd and if $F\in\mathcal{B}_{sym}^{0}(q)$
is of the form $F=(f_{1}\otimes\dots\otimes f_{q})_{\text{sym}}$,
then%
\[
\]
\[
\epsilon^{-q/2}\E(F(B_{T}^{(1)},\dots,B_{T}^{(q)})\1_{\#\mathcal{L}_{1}=q})\underset{\epsilon\rightarrow0}{\longrightarrow}0\,.
\]
\end{cor}

\begin{proof}
We take $\alpha\in\left(\frac{q}{2}\left\lceil \frac{q}{2}\right\rceil ^{-1},1\right)$.
We can decompose
\begin{multline*}
\epsilon^{-q/2}\left|\E(F(B_{T}^{(1)},\dots,B_{T}^{(q)})\1_{\#\mathcal{L}_{1}=q})\right|\\
=|\epsilon^{-q/2}\sum_{k\in[q]}\sum_{l\in\{0,1,\dots,(q-k-1)_{+}\}}\E(\1_{C_{k,l}(\alpha T)}\1_{\#\mathcal{L}_{1}=q}F(B_{T}^{(1)},\dots,B_{T}^{(q)}))\\
+\epsilon^{-q/2}\E(\1_{L_{\alpha T}\geq\left\lceil q/2\right\rceil }\1_{\#\mathcal{L}_{1}=q}F(B_{T}^{(1)},\dots,B_{T}^{(q)}))|\\
\text{(by Lemmas \ref{lem:conv-nouille-libre}, \ref{lem:conv-nouilles-liees}) }\underset{\epsilon\rightarrow0}{\longrightarrow}0\,.
\end{multline*}
\end{proof}
\begin{cor}
\label{cor:maj-fonction-centree}Suppose $F\in\mathcal{B}_{sym}^{0}(q)$
is of the form $F=(f_{1}\otimes\dots\otimes f_{q})_{\text{sym}}$.
Let $A$ in $\sigma(\mathcal{L}_{2})$. Then
\[
|\E(F(B_{T}^{(1)},\dots,B_{T}^{(q)})\1_{A})|\leq\Vert F\Vert_{\infty}\epsilon^{q/2}\left\{ K_{1}(q)+\Gamma_{1}^{q}C_{\text{tree}}(q)\left(\frac{1}{\delta}\right)^{q}q^{2}\right\} \,
\]
\end{cor}

\begin{proof}
From Lemmas \ref{lem:conv-nouille-libre}, \ref{lem:conv-nouilles-liees},
we get 
\begin{multline*}
|\E(F(B_{T}^{(1)},\dots,B_{T}^{(q)})\1_{A})|=|\E(F(B_{T}^{(1)},\dots,B_{T}^{(q)})\1_{A}(\1_{L_{\alpha T}\geq q/2}+\sum_{k'\in[q]}\sum_{0\leq l\leq(q-k'-1)_{+}}\1_{C_{k',l}(\alpha T)}))|\\
\leq\Vert F\Vert_{\infty}\epsilon^{q/2}\left\{ K_{1}(q)+\Gamma_{1}^{q}C_{\text{tree}}(q)\left(\frac{1}{\delta}\right)^{q}\sum_{k'\in[q]}(q-k'-1)_{+}\right\} \,.
\end{multline*}
\end{proof}
We now want to find the limit of $\epsilon^{-q/2}\E(\1_{L_{T}\leq q/2}\1_{\#\mathcal{L}_{1}=q}F(B_{T}^{(1)},\dots,B_{T}^{(q)}))$
when $\epsilon$ goes to $0$, for $q$ even. First we need a technical
lemma. 

For any $i$, the process $(B_{t}^{(i)})$ has a stationary law (see
Theorem 3.3 p. 151 of \cite{asmussen-2003}). Let $B_{\infty}$ be
a random variable having this stationary law $\eta$ (it has already
appeared in Section \ref{sec:Rate-of-convergence}). We can always
suppose that it is independent of all the other variables.
\begin{lem}
\label{lem:terme-b} Let $f_{1}$ , $f_{2}$ be in $\mathcal{B}_{sym}^{0}(1)$.
Let $\alpha$ belong to $(0,1)$. We have
\[
\int_{-\infty}^{-\log(\delta)}e^{-v}|\E(f_{1}(\overline{B}_{0}^{(1),v})f_{2}(\overline{B}_{0}^{(2),v}))|dv<\infty
\]
and 
\begin{multline*}
\left|e^{T-\alpha T-B_{\alpha T}^{(1)}}\E(f_{1}\otimes f_{2}(B_{T,}^{(1)}B_{T}^{(2)})\1_{P_{1,2}(T)^{\complement}}|\mathcal{F}_{S(\alpha T)},P_{1,2}(\alpha T))\right.\\
\left.-\int_{-\infty}^{-\log(\delta)}e^{-v}\E(\1_{v\leq\overline{B}_{0}^{(1),v}}f_{1}(\overline{B}_{0}^{(1),v})f_{2}(\overline{B}_{0}^{(2),v}))dv\right|\\
\leq\Gamma_{2}\Vert f_{1}\Vert_{\infty}\Vert f_{2}\Vert_{\infty}\exp\left(-(T-\alpha T)\left(\frac{\theta-\epsilon'-1}{2}\right)\right)\,,
\end{multline*}
where 
\[
\Gamma_{2}=\frac{\Gamma_{1}^{2}}{\delta^{2+2(\theta-\epsilon')}(2(\theta-\epsilon')-1)}+\frac{\Gamma_{1}}{\delta^{\theta-\epsilon'}}+\frac{\Gamma_{1}^{2}}{\delta^{2(\theta-\epsilon')}(2(\theta-\epsilon')-1)}\,.
\]
 
\end{lem}

\begin{proof}
From now on, we suppose that $\alpha T-\log(\delta)<(T+\alpha T)/2$,
$(T+\alpha T)/2-\log(\delta)<T$ (this is true if $T$ is large enough).
We have, for all $s$ in $[\alpha T+B_{\alpha T}^{(1)},T]$,

\[
\p(u\{s,2\}=u\{s,1\}|\mathcal{F}_{S(\alpha T)},P_{1,2}(\alpha T),(\xi_{u\{t,1\}})_{0\leq t\leq T})=\exp(-(s+B_{s}^{(1)}-(\alpha T+B_{\alpha T}^{(1)}))\,.
\]
And so, 
\begin{multline*}
\E(f_{1}\otimes f_{2}(B_{T,}^{(1)}B_{T}^{(2)})\1_{P_{1,2}(T)^{\complement}}|\mathcal{F}_{S(\alpha T)},P_{1,2}(\alpha T))\\
=\E\left(\E(f_{1}\otimes f_{2}(B_{T,}^{(1)}B_{T}^{(2)})\1_{P_{1,2}(T)^{\complement}}|\mathcal{F}_{S(\alpha T)},P_{1,2}(\alpha T),(\xi_{u\{t,1\}})_{0\leq t\leq T})|\mathcal{F}_{S(\alpha T)},P_{1,2}(\alpha T)\right)\\
\text{(keep in mind that }\widehat{B}^{(1),v}=B^{(1)}\text{ for all }v\text{)}\\
=\E(\E(\int_{\alpha T+B_{\alpha T}^{(1)}}^{T+B_{T}^{(1)}}e^{-(v-\alpha T-\widehat{B}_{\alpha T}^{(1),v})}f_{1}(\widehat{B}_{T}^{(1),v})f_{2}(\widehat{B}_{T}^{(2),v})dv|\mathcal{F}_{S(\alpha T)},P_{1,2}(\alpha T),(\xi_{u\{t,1\}})_{0\leq t\leq T})\\
|\mathcal{F}_{S(\alpha T)},P_{1,2}(\alpha T))\\
=\E(\int_{\alpha T+B_{\alpha T}^{(1)}}^{T+B_{T}^{(1)}}e^{-(v-\alpha T-\widehat{B}_{\alpha T}^{(1)v})}f_{1}(\widehat{B}_{T}^{(1),v})f_{2}(\widehat{B}_{T}^{(2),v})dv|\mathcal{F}_{S(\alpha T)},P_{1,2}(\alpha T))
\end{multline*}
We have 
\begin{multline}
\left|e^{T-\alpha T-B_{\alpha T}^{(1)}}\E\left(\int_{\alpha T+B_{\alpha T}^{(1)}}^{(T+\alpha T)/2}e^{-(v-\alpha T-\widehat{B}_{\alpha T}^{(1),v})}f_{1}(\widehat{B}_{T}^{(1),v})f_{2}(\widehat{B}_{T}^{(2),v})dv|\mathcal{F}_{S(\alpha T)},P_{1,2}(\alpha T)\right)\right|\\
=e^{T-\alpha T-B_{\alpha T}^{(1)}}|\E(\int_{\alpha T+B_{\alpha T}^{(1)}}^{(T+\alpha T)/2}e^{-(v-\alpha T-\widehat{B}_{\alpha T}^{(1),v})}\\
\times\E(f_{1}(\widehat{B}_{T}^{(1),v})f_{2}(\widehat{B}_{T}^{(2),v})|\widehat{B}_{v}^{(1),v},\widehat{B}_{v}^{(2),v},\mathcal{F}_{S(\alpha T)},P_{1,2}(\alpha T))dv\\
|\mathcal{F}_{S(\alpha T)},P_{1,2}(\alpha T))|\\
\text{(using the fact that }\widehat{B}_{T}^{(1),v}\text{ and }\widehat{B}_{T}^{(2),v}\text{ are independant }\\
\text{conditionally to }\widehat{B}_{v}^{(1),v},\widehat{B}_{v}^{(2),v},\mathcal{F}_{S(\alpha T)},P_{1,2}(\alpha T)\text{,}\text{ if }T\geq v-\log(\delta)\text{, }\\
\text{we get, by Theorem \ref{lem:sgibnev} and Corollary \ref{cor:sgibnev})}\\
\leq e^{T-\alpha T-B_{\alpha T}^{(1)}}\\
\times\E(\int_{\alpha T+B_{\alpha T}^{(1)}}^{(T+\alpha T)/2}e^{-(v-\alpha T-\widehat{B}_{\alpha T}^{(1)})}(\Gamma_{1}\Vert f_{1}\Vert_{\infty}e^{-(\theta-\epsilon')(T-v-\widehat{B}_{v}^{(1),v})_{+}}\times\Gamma_{1}\Vert f_{2}\Vert_{\infty}e^{-(\theta-\epsilon')(T-v-\widehat{B}_{v}^{(2),v})_{+}})dv\\
|\mathcal{F}_{S(\alpha T)},P_{1,2}(\alpha T))\\
\leq\Gamma_{1}^{2}\Vert f_{1}\Vert_{\infty}\Vert f_{2}\Vert_{\infty}e^{T-\alpha T-\log(\delta)}\int_{\alpha T}^{(T+\alpha T)/2}e^{-(v-\alpha T+\log(\delta))}e^{-2(\theta-\epsilon')(T-v+\log(\delta))}dv\\
=\frac{\Gamma_{1}^{2}\Vert f_{1}\Vert_{\infty}\Vert f_{2}\Vert_{\infty}}{\delta^{2+2(\theta-\epsilon')}}e^{T-2(\theta-\epsilon')T}\left[\frac{e^{(2(\theta-\epsilon')-1)v}}{2(\theta-\epsilon')-1}\right]_{\alpha T}^{(T+\alpha T)/2}\\
\leq\frac{\Gamma_{1}^{2}\Vert f_{1}\Vert_{\infty}\Vert f_{2}\Vert_{\infty}}{\delta^{2+2(\theta-\epsilon')}}\frac{\exp\left(-(2(\theta-\epsilon')-1)T+(2(\theta-\epsilon')-1)\frac{(T+\alpha T)}{2}\right)}{2(\theta-\epsilon')-1}\\
=\frac{\Gamma_{1}^{2}\Vert f_{1}\Vert_{\infty}\Vert f_{2}\Vert_{\infty}}{\delta^{2+2(\theta-\epsilon')}}\frac{\exp\left(-(2(\theta-\epsilon')-1)\left(\frac{T-\alpha T}{2}\right)\right)}{2(\theta-\epsilon')-1}\,.\label{eq:bout-01}
\end{multline}
We have
\begin{multline}
\left|e^{T-\alpha T-B_{\alpha T}^{(1)}}\E\left(\int_{(T+\alpha T)/2}^{T+B_{T}^{(1)}}e^{-(v-\alpha T-B_{\alpha T}^{(1)})}f_{1}(\widehat{B}_{T}^{(1),v})f_{2}(\widehat{B}_{T}^{(2),v})dv|\mathcal{F}_{S(\alpha T)},P_{1,2}(\alpha T)\right)\right.\\
\left.-\int_{(T+\alpha T)/2}^{T-\log(\delta)}e^{-(v-T)}\E(\1_{v\leq T+\overline{B}_{T}^{(1),v}}f_{1}(\overline{B}_{T}^{(1),v})f_{2}(\overline{B}_{T}^{(2),v}))dv\right|\\
=\left|e^{T-\alpha T-B_{\alpha T}^{(1)}}\E(\int_{(T+\alpha T)/2}^{T-\log(\delta)}e^{-(v-\alpha T-B_{\alpha T}^{(1)})}\1_{v\leq T+B_{T}^{(1)}}f_{1}(\widehat{B}_{T}^{(1),v})f_{2}(\widehat{B}_{T}^{(2),v})dv|\mathcal{F}_{S(\alpha T)},P_{1,2}(\alpha T))\right.\\
\left.-e^{T-\alpha T-B_{\alpha T}^{(1)}}\E(\int_{(T+\alpha T)/2}^{T-\log(\delta)}e^{-(v-\alpha T-B_{\alpha T}^{(1)})}\1_{v\leq T-\overline{B}_{T}^{(1),v}}f_{1}(\overline{B}_{T}^{(1),v})f_{2}(\overline{B}_{T}^{(2),v})dv|\mathcal{F}_{S(\alpha T)},P_{1,2}(\alpha T))\right|\\
=e^{T-\alpha T-B_{\alpha T}^{(1)}}\left|\int_{(T+\alpha T)/2}^{T-\log(\delta)}e^{-(v-\alpha T-B_{\alpha T}^{(1)})}\E(\E(\1_{v\leq T+B_{T}^{(1)}}f_{1}(\widehat{B}_{T}^{(1),v})f_{2}(\widehat{B}_{T}^{(2),v})\right.\\
|\widehat{B}_{v}^{(1),v},\mathcal{F}_{S(\alpha T)},P_{1,2}(\alpha T))|\mathcal{F}_{S(\alpha T)},P_{1,2}(\alpha T))dv\\
\left.-\int_{(T+\alpha T)/2}^{T-\log(\delta)}e^{-(v-\alpha T-B_{\alpha T}^{(1)})}\E(\E(\1_{v\leq T+\overline{B}_{T}^{(1),v}}f_{1}(\overline{B}_{T}^{(1),v})f_{2}(\overline{B}_{T}^{(2),v}))dv|\overline{B}_{v}^{(1),v}))dv\right|\label{eq:abs-a-borner-01}
\end{multline}
We observe that, for all $v$ in $[(T+\alpha T)/2,T-\log(\delta)]$,
\[
\E(\1_{v\leq T+B_{T}^{(1)}}f_{1}(\widehat{B}_{T}^{(1),v})f_{2}(\widehat{B}_{T}^{(2),v})|\widehat{B}_{v}^{(1),v},\mathcal{F}_{S(\alpha T)},P_{1,2}(\alpha T))=\Psi(\widehat{B}_{v}^{(1),v})\,,
\]
\[
\E(\1_{v\leq T+\overline{B}_{T}^{(1),v}}f_{1}(\overline{B}_{T}^{(1),v})f_{2}(\overline{B}_{T}^{(2),v})|\overline{B}_{v}^{(1),v})=\Psi(\overline{B}_{v}^{(1),v})\overset{\text{law}}{=}\Psi(B_{\infty})\,,
\]
for some function $\Psi$ (the same on both lines) such that $\Vert\Psi\Vert_{\infty}\leq\Vert f_{1}\Vert_{\infty}\Vert f_{2}\Vert_{\infty}$.
So, by Theorem \ref{lem:sgibnev} and Corollary \ref{cor:sgibnev},
the quantity in Equation (\ref{eq:abs-a-borner-01}) can be bounded
by 
\[
e^{T-\alpha T-B_{\alpha T}^{(1)}}\int_{(T+\alpha T)/2}^{T-\log(\delta)}e^{-(v-\alpha T-B_{\alpha T}^{(1)})}\Gamma_{1}\Vert f_{1}\Vert_{\infty}\Vert f_{2}\Vert_{\infty}e^{-(\theta-\epsilon')(v-\alpha T-B_{\alpha T}^{(1)})}dv
\]
(coming from Corollary \ref{cor:sgibnev} there is an integral over
a set of Lebesgue measure zero in the above bound, but this term vanishes).
The above bound can in turn be bounded by:
\begin{multline}
\frac{\Gamma_{1}\Vert f_{1}\Vert_{\infty}\Vert f_{2}\Vert_{\infty}}{\delta^{(\theta-\epsilon')}}e^{T}\int_{(T+\alpha T)/2}^{T-\log(\delta)}e^{(\theta-\epsilon')\alpha T}e^{-(\theta-\epsilon'+1)v}dv\\
\leq\frac{\Gamma_{1}\Vert f_{1}\Vert_{\infty}\Vert f_{2}\Vert_{\infty}}{\delta^{\theta-\epsilon'}}e^{T+\alpha T(\theta-\epsilon')}\exp(-(\theta-\epsilon'+1)(\frac{T+\alpha T}{2}))\\
=\frac{\Gamma_{1}\Vert f_{1}\Vert_{\infty}\Vert f_{2}\Vert_{\infty}}{\delta^{\theta-\epsilon'}}\exp\left(-(\theta-\epsilon'-1)\left(\frac{T-\alpha T}{2}\right)\right)\,.\label{eq:bout-02}
\end{multline}
We have
\begin{multline}
\int_{\frac{T+\alpha T}{2}}^{T-\log(\delta)}e^{-(v-T)}\E(\1_{v\leq T+\overline{B}_{T}^{(1),v}}f_{1}(\overline{B}_{T}^{(1),v})f_{2}(\overline{B}_{T}^{(2),v}))dv\\
=\E\left(\int_{-\left(\frac{T-\alpha T}{2}\right)}^{-\log(\delta)}e^{-v}\1_{v\leq\overline{B}_{0}^{(1),v}}f_{1}(\overline{B}_{0}^{(1),v})f_{2}(\overline{B}_{0}^{(1),v})dv\right)\label{eq:bout-03}
\end{multline}
and
\begin{multline}
\int_{-\infty}^{-\frac{(T-\alpha T)}{2}}e^{-v}|\E(f_{1}(\overline{B}_{0}^{(1),v})f_{2}(\overline{B}_{0}^{(2),v}))|dv\\
\text{(since }\overline{B}_{0}^{(1),v}\text{ and }\overline{B}_{0}^{(2),v}\text{ are independant conditionnaly on }\overline{B}_{v}^{(1),v}\text{, }\overline{B}_{v}^{(2),v}\\
\text{if }v-\log(\delta)\leq0\text{)}\\
\text{(using Theorem \ref{lem:sgibnev} and Corollary \ref{cor:sgibnev}) }\\
\leq\int_{-\infty}^{-\frac{(T-\alpha T)}{2}}e^{-v}\Gamma_{1}^{2}\Vert f_{1}\Vert_{\infty}\Vert f_{2}\Vert_{\infty}\E(e^{-(\theta-\epsilon')(-v-\overline{B}_{v}^{(1),v})_{+}}e^{-(\theta-\epsilon')(-v-\overline{B}_{v}^{(2),v})_{+}})dv\\
\text{(again, coming from Corollary \ref{cor:sgibnev} there is an integral }\\
\text{over a set of Lebesgue measure zero in the above bound, but this term vanishes)}\\
\leq\int_{-\infty}^{-\frac{(T-\alpha T)}{2}}e^{-v}\Gamma_{1}^{2}\Vert f_{1}\Vert_{\infty}\Vert f_{2}\Vert_{\infty}e^{-2(\theta-\epsilon')(-v+\log(\delta))}dv\\
=\frac{\Gamma_{1}^{2}\Vert f_{1}\Vert_{\infty}\Vert f_{2}\Vert_{\infty}}{\delta^{2(\theta-\epsilon')}}\frac{\exp\left(-(2(\theta-\epsilon')-1)\frac{(T-\alpha T)}{2}\right)}{2(\theta-\epsilon')-1}\,.\label{eq:bout-04}
\end{multline}
Equations (\ref{eq:bout-01}), (\ref{eq:bout-02}), (\ref{eq:bout-03})
and (\ref{eq:bout-04}) give us the desired result.
\end{proof}
\begin{lem}
\label{lem:calcul-exact-limite}Let $k$ in $\{0,1,2,\dots,p\}$.
We suppose $q$ is even and $q=2p$. Let $\alpha\in(q/(q+2),1)$.
We suppose $F=f_{1}\otimes f_{2}\otimes\dots\otimes f_{q}$, with
$f_{1}$, \ldots{} , $f_{q}$ in $\mathcal{B}_{sym}^{0}(1)$. We
then have~:
\begin{multline}
\epsilon^{-q/2}\E(F(B_{T,}^{(1)}\dots,B_{T}^{(q)})\1_{P_{\alpha T}}\1_{\#\mathcal{L}_{1}=q})\\
\underset{\epsilon\rightarrow0}{\longrightarrow}\prod_{i=1}^{p}\int_{-\infty}^{-\log(\delta)}e^{-v}\E(\1_{v\leq\overline{B}_{0}^{(1),v}}f_{2i-1}(\overline{B}_{0}^{(1),v})f_{2i}(\overline{B}_{0}^{(2),v}))dv\,.\label{eq:conv-01}
\end{multline}
\end{lem}

\begin{proof}
We have 
\[
P_{\alpha T}\cap\{\#\mathcal{L}_{1}=q\}=P_{\alpha T}\cap\underset{1\leq i\leq p}{\bigcap}P_{2i-1,2i}(T)^{\complement}\,.
\]
By Lemma \ref{lem:terme-b}, we have, for some constant $C$,
\begin{multline}
\left|e^{pT}\E\left(\left.\1_{P_{\alpha T}}\prod_{i=1}^{p}\E(f_{2i-1}\otimes f_{2i}(B_{T}^{(2i-1)},B_{T}^{(2i)}))\1_{P_{2i-1,2i}(T)^{\complement}}\right|\mathcal{F}_{S(\alpha T)}\right)\right.\\
-\E\left.\left(\1_{P_{\alpha T}}\left.\prod_{i=1}^{p}e^{B_{\alpha T}^{(2i-1)}+\alpha T}\int_{-\infty}^{-\log(\delta)}e^{-v}\E(\1_{v\leq\overline{B}_{0}^{(1),v}}f_{2i-1}(\overline{B}_{0}^{(1),v})f_{2i}(\overline{B}_{0}^{(2),v}))dv\right)\right.\right|\\
\leq\prod_{i=1}^{p}\left(\Gamma_{2}\Vert f_{2i-1}\Vert_{\infty}\Vert f_{2i}\Vert_{\infty}\right)\times\E\left(\1_{P_{\alpha T}}\prod_{i=1}^{p}[e^{B_{\alpha T}^{(2i-1)}+\alpha T}e^{-(T-\alpha T)\frac{(\theta-\epsilon'-1)}{2}}]\right)\,.\label{eq:maj-02}
\end{multline}
We introduce the events (for $t\in[0,T]$)
\[
O_{t}=\left\{ \#\{u\{t,2i-1\},1\leq i\leq p\}=p\right\} \,,
\]
and the tribes (for $i$ in $[q]$, $t\in[0,T]$)
\[
\mathcal{F}_{t,i}=\sigma(u\{t,i\},\xi_{u\{t,i\}})\,.
\]
We have~:
\begin{multline}
\E(\1_{P_{\alpha T}}\prod_{i=1}^{p}e^{B_{\alpha T}^{(2i-1)}+\alpha T})=\E(\1_{O_{\alpha T}}\prod_{i=1}^{p}e^{B_{\alpha T}^{(2i-1)}+\alpha T}\E(\prod_{i=1}^{p}\1_{u\{\alpha T,2i-1\}=u\{\alpha T,2i\}}|\vee_{1\leq i\leq p}\mathcal{F}_{\alpha T,2i-1}))\\
=\E(\1_{O_{\alpha T}})\,.\label{eq:1_O}
\end{multline}
We then observe that
\[
O_{\alpha T}^{\complement}=\cup_{i\in[p]}\cup_{j\in[p],j\neq i}\{u\{\alpha T,2i-1\}=u\{\alpha T,2j-1\}\}\,,
\]
and, for $i\neq j$,
\begin{eqnarray*}
\p(u\{\alpha T,2i-1\}=u\{\alpha T,2j-1\}) & = & \E(\E(\1_{u\{\alpha T,2i-1\}=u\{\alpha T,2j-1\}}|\mathcal{F}_{\alpha T,2i-1}))\\
 & = & \E(e^{-\alpha T-B_{\alpha T}^{(2i-1)}})\\
\mbox{(because of Assumption (\ref{hyp:delta-step}))} & \leq & \E(e^{-\alpha T-\log(\delta)})\,.
\end{eqnarray*}
So
\[
\p(O_{\alpha T})\underset{\epsilon\rightarrow0}{\longrightarrow}1\,.
\]
This finishes the proof of Equation (\ref{eq:conv-01}).

\end{proof}

\subsection{Convergence result}

For $f$ and $g$ bounded measurable functions, we set
\begin{equation}
V(f,g)=\int_{-\infty}^{-\log(\delta)}e^{-v}\E(\1_{v\leq\overline{B}_{0}^{(1),v}}f(\overline{B}_{0}^{(1),v})g(\overline{B}_{0}^{(2),v}))dv\,.\label{eq:def-V}
\end{equation}
For $q$ even, we set $\mathcal{I}_{q}$ to be the set of partitions
of $[q]$ into subsets of cardinality $2$. For $I$ in $\mathcal{I}_{q}$
and $t$ in $[0,T]$, we introduce
\[
P_{t,I_{q}}=\{\forall\{i,j\}\in I\,,\,\exists u\in\mathcal{U}\text{ such that }\xi_{u}<e^{-t}\,,\,\xi_{\boldsymbol{a}(u)}\geq e^{-t}\,,\,A_{u}=\{i,j\}\}\,.
\]
For $t$ in $[0,T]$, we define
\[
\mathcal{P}_{t}=\cup_{I\in\mathcal{I}_{q}}P_{t,I}\,.
\]
\begin{prop}
\label{prop:conv-q-pair}Let $q$ be in $\N^{*}$. Let $F=(f_{1}\otimes\dots\otimes f_{q})_{\text{sym }}$
with $f_{1}$, \dots , $f_{q}$ in $\Bsym(1)$. If $q$ is even ($q=2p$)
then
\begin{equation}
\epsilon^{q/2}\E(F(B_{T}^{(1)},\dots,B_{T}^{(q)})\1_{\#\mathcal{L}_{1}=q})\underset{\epsilon\rightarrow0}{\longrightarrow}\sum_{I\in\mathcal{I}_{q}}\prod_{\{a,b\}\in I}V(f_{a},f_{b})\,.\label{eq:conv-q-even}
\end{equation}
\end{prop}

\begin{proof}
Let $\alpha$ be in $(q/(q+2),1)$. We have
\begin{eqnarray*}
\epsilon^{-q/2}\E(F(B_{T}^{(1)},\dots,B_{T}^{(q)})\1_{\#\mathcal{L}_{1}=q}) & = & \epsilon^{-q/2}\E(F(B_{T}^{(1)},\dots,B_{T}^{(q)})\1_{\#\mathcal{L}_{1}=q}(\1_{\mathcal{P}_{\alpha T}}+\1_{\mathcal{P}_{\alpha T}^{\complement}}))\,.
\end{eqnarray*}
By Lemma \ref{lem:conv-nouille-libre} and Lemma \ref{lem:conv-nouilles-liees},
we have that
\[
\lim_{\epsilon\rightarrow0}\epsilon^{-q/2}\E(F(B_{T}^{(1)},\dots,B_{T}^{(q)})\1_{\#\mathcal{L}_{1}=q}\1_{\mathcal{P}_{\alpha T}^{\complement}})=0
\]
(because $(B_{T}^{(1)},\dots,B_{T}^{(q)})$ is exchangeable). We compute~:
\begin{multline*}
\epsilon^{-q/2}\E(F(B_{T}^{(1)},\dots,B_{T}^{(q)})\1_{\#\mathcal{L}_{1}=q}\1_{\mathcal{P}_{\alpha T}})=\epsilon^{-q/2}\E(F(B_{T}^{(1)},\dots,B_{T}^{(q)})\1_{\#\mathcal{L}_{1}=q}\sum_{I_{q}\in\mathcal{I}_{q}}\1_{P_{\alpha T,I_{q}}})\\
\text{(as }F\text{ is symmetric and }(B_{T}^{(1)},\dots,B_{T}^{(q)})\text{ is exchangeable)}\\
=\frac{q!}{2^{q/2}\left(\frac{q}{2}\right)!}\epsilon^{-q/2}\E(F(B_{T}^{(1)},\dots,B_{T}^{(q)})\1_{\#\mathcal{L}_{1}=q}\1_{P_{\alpha T}})\\
=\frac{q!\epsilon^{-q/2}}{2^{q/2}\left(\frac{q}{2}\right)!}\frac{1}{q!}\sum_{\sigma\in\mathcal{S}_{q}}\E((f_{\sigma(1)}\otimes\dots\otimes f_{\sigma(q)})(B_{T}^{(1)},\dots,B_{T}^{(q)})\1_{\#\mathcal{L}_{1}=q}\1_{P_{\alpha T}})\\
\text{(by Lemma \ref{lem:calcul-exact-limite}) }\\
\underset{\epsilon\rightarrow0}{\longrightarrow}\frac{1}{2^{q/2}\left(\frac{q}{2}\right)!}\sum_{\sigma\in\mathcal{S}_{q}}\prod_{i=1}^{p}V(f_{\sigma(2i-1)},f_{\sigma(2i)})=\sum_{I\in\mathcal{I}_{q}}\prod_{\{a,b\}\in I}V(f_{a},f_{b})\,.
\end{multline*}
\begin{eqnarray*}
\end{eqnarray*}
\end{proof}

\section{\label{sec:Results}Results}

We are interested in the probability measure $\gamma_{T}$ defined
by its action on bounded measurable functions $F\,:\,[0,1]\rightarrow\R$
by
\[
\gamma_{T}(F)=\sum_{u\in\mathcal{U}_{\epsilon}}X_{u}F\left(\frac{X_{u}}{\epsilon}\right)\,.
\]
We define, for all $q$ in $\N^{*}$, $F$ from $[0,1]^{q}$ to $\R$
,
\[
\gamma_{T}^{\otimes q}(F)=\sum_{a\,:\,[q]\rightarrow\mathcal{U_{\epsilon}}}X_{a(1)}\dots X_{a(q)}F\left(\frac{X_{a(1)}}{\epsilon},\dots,\frac{X_{a(q)}}{\epsilon}\right)\,,
\]

\[
\gamma_{T}^{\odot q}(F)=\sum_{a\,:\,[q]\hookrightarrow\mathcal{U_{\epsilon}}}X_{a(1)}\dots X_{a(q)}F\left(\frac{X_{a(1)}}{\epsilon},\dots,\frac{X_{a(q)}}{\epsilon}\right)\,,
\]
where the last sum is taken over all the injective applications $a$
from $[q]$ to $\Ue$. If we set 
\[
\Phi(F)\,:\,(y_{1},\dots,y_{q})\in\R^{+}\mapsto F(e^{-y_{1}},\dots,e^{-y_{q}})\,,
\]
then
\[
\E(\gamma_{T}^{\otimes q}(F))=\E(\Phi(F)(B_{T}^{(1)},\dots,B_{T}^{(q)}))\,,
\]

\[
\E(\gamma_{T}^{\odot q}(F))=\E(\Phi(F)(B_{T}^{(1)},\dots,B_{T}^{(q)})\1_{\#\mathcal{L}_{1}}=q)\,.
\]
We define, for all bounded continuous $f\,:\,\R^{+}\rightarrow\R$,
\begin{equation}
\gamma_{\infty}(f)=\eta(\Phi(f))\,.\label{eq:def-gamma-infty}
\end{equation}
\begin{prop}
[Law of large numbers] \label{prop:convergence-ps}Let f be a continuous
function from $[0,1]$ to $\R$. We have:
\[
\gamma_{T}(f)\underset{T\rightarrow+\infty}{\overset{\text{a.s.}}{\longrightarrow}}\gamma_{\infty}(f)\,.
\]
\end{prop}

\begin{proof}
We take a bounded measurable function $f\,:\,[0,1]\rightarrow\R$.
We define $\overline{f}=f-\eta(\Phi(f))$. We take an integer $q\geq2$.
We introduce the notation~:
\[
\forall g\,:\,\R^{+}\rightarrow\R\,,\,\forall(x_{1},\dots,x_{q})\in\R^{q}\,,\,g^{\otimes q}(x_{1},\dots,x_{q})=g(x_{1})g(x_{2})\dots g(x_{q})\,.
\]
 We have
\begin{eqnarray*}
\E((\gamma_{T}(f)-\eta(\Phi(f)))^{q}) & = & \E((\gamma_{T}(\overline{f}))^{q})\\
 & = & \E(\gamma_{t}^{\otimes q}(\overline{f}^{\otimes q}))\\
 & = & \E(\gamma_{t}^{\otimes q}((\overline{f}^{\otimes q})_{\text{sym}}))\\
\text{(by Corollary \ref{cor:maj-fonction-centree})} & \leq & \Vert\overline{f}\Vert_{\infty}^{q}\epsilon^{q/2}\left\{ K_{1}(q)+\Gamma_{1}^{q}C_{\text{tree}}(q)\left(\frac{1}{\delta}\right)^{q}q^{2}\right\} \,.
\end{eqnarray*}
We now take sequences $(T_{n}=-\log(n))_{n\geq1}$, $(\epsilon_{n}=1/n)_{n\geq1}$.
We then have, for all $n$ and for all $\iota>0$,
\[
\p([\gamma_{T_{n}}(f)-\eta(\Phi(f))]^{4}\geq\iota)\leq\frac{\Vert\overline{f}\Vert_{\infty}^{4}}{\iota n^{2}}\left\{ K_{1}(4)+\Gamma_{1}^{4}C_{\text{tree}}(4)\left(\frac{1}{\delta}\right)^{4}\times16\right\} \,.
\]
So, by Borell-Cantelli's Lemma,
\begin{equation}
\gamma_{T_{n}}(f)\underset{n\rightarrow+\infty}{\underset{\longrightarrow}{\text{a.s.}}}\eta(\Phi(f))\,.\label{eq:conv-discrete}
\end{equation}

Let $n$ be in $\N^{*}$. We can decompose 
\[
\mathcal{U}_{\epsilon_{n}}=\mathcal{U}_{\epsilon_{n}}^{(1)}\sqcup\mathcal{U}_{\epsilon_{n}}^{(2)}\text{ where }\mathcal{U}_{\epsilon_{n}}^{(1)}=\mathcal{U}_{\epsilon_{n}}\cap\mathcal{U}_{\epsilon_{n+1}}\,,\,\mathcal{U}_{\epsilon_{n}}^{(2)}=\mathcal{U}_{\epsilon_{n}}\backslash\mathcal{U}_{\epsilon_{n+1}}\,.
\]
For $u$ in $\mathcal{U}_{\epsilon_{n}}\backslash\mathcal{U}_{\epsilon_{n+1}}$,
we set $\boldsymbol{d}(u)=\{v\in\mathcal{U}_{\epsilon_{n+1}}\,:\,\boldsymbol{a}(v)=u\}$.
We can then write
\begin{eqnarray*}
\sum_{u\in\mathcal{U}_{\epsilon_{n}}}X_{u}f\left(\frac{X_{u}}{\epsilon_{n}}\right) & = & \sum_{u\in\mathcal{U}_{\epsilon_{n}}^{(1)}}X_{u}f(nX_{u})+\sum_{u\in\mathcal{U}_{\epsilon_{n}}^{(2)}}X_{u}f(nX_{u})\,,\\
\sum_{u\in\mathcal{U}_{\epsilon_{n+1}}}X_{u}f\left(\frac{X_{u}}{\epsilon_{n+1}}\right) & = & \sum_{u\in\mathcal{U}_{\epsilon_{n}}^{(1)}}X_{u}f((n+1)X_{u})+\sum_{u\in\mathcal{U}_{\epsilon_{n}}^{(2)}}\sum_{v\in\boldsymbol{d}(u)}X_{v}f((n+1)X_{v})\,.
\end{eqnarray*}
 So we have, for all $n$,
\begin{multline*}
\left|\sum_{u\in\mathcal{U}_{\epsilon_{n}}^{(2)}}\sum_{v\in\boldsymbol{d}(u)}X_{v}f((n+1)X_{v})-\sum_{u\in\mathcal{U}_{\epsilon_{n}}^{(2)}}X_{u}f(nX_{u})\right|\\
\leq|\gamma_{T_{n+1}}(f)-\gamma_{T_{n}}(f)|+\left|\sum_{u\in\mathcal{U}_{\epsilon_{n}}^{(1)}}X_{u}f((n+1)X_{u})-\sum_{u\in\mathcal{U}_{\epsilon_{n}}^{(1)}}X_{u}f(nX_{u})\right|\,.
\end{multline*}
If we take $f=\Id$, the two terms in the equation above can be transformed:
\begin{multline*}
\left|\sum_{u\in\mathcal{U}_{\epsilon_{n}}^{(2)}}\sum_{v\in\boldsymbol{d}(u)}X_{v}f((n+1)X_{v})-\sum_{u\in\mathcal{U}_{\epsilon_{n}}^{(2)}}X_{u}f(nX_{u})\right|\\
\geq\left|\sum_{u\in\mathcal{U}_{\epsilon_{n}}^{(2)}}\left(X_{u}f(nX_{u})-\sum_{v\in\boldsymbol{d}(u)}X_{v}f(nX_{v})\right)\right|-\left|\sum_{u\in\mathcal{U}_{\epsilon_{n}}^{(2)}}\sum_{v\in\boldsymbol{d}(u)}(X_{v}f(nX_{v})-X_{v}f((n+1)X_{v}))\right|\\
\text{(by Assumption \ref{hyp:delta-step})}\\
\geq\sum_{u\in\mathcal{U}_{\epsilon_{n}}^{(2)}}\left(X_{u}f(nX_{u})-\sum_{v\in\boldsymbol{d}(u)}X_{v}f(nX_{u})e^{-a}\right)-\left|\sum_{u\in\mathcal{U}_{\epsilon_{n}}^{(2)}}\sum_{v\in\boldsymbol{d}(u)}(X_{v}f(nX_{v})-X_{v}f((n+1)X_{v}))\right|\\
\geq\sum_{u\in\mathcal{U}_{\epsilon_{n}}^{(2)}}X_{u}(1-e^{-a})\frac{n}{n+1}-\sum_{u\in\mathcal{U}_{\epsilon_{n}}^{(2)}}\sum_{v\in\boldsymbol{d}(u)}X_{v}\frac{1}{n+1}\,,
\end{multline*}
\begin{multline*}
|\gamma_{T_{n+1}}(f)-\gamma_{T_{n}}(f)|+\left|\sum_{u\in\mathcal{U}_{\epsilon_{n}}^{(1)}}X_{u}f((n+1)X_{u})-\sum_{u\in\mathcal{U}_{\epsilon_{n}}^{(1)}}X_{u}f(nX_{u})\right|\\
\leq|\gamma_{T_{n+1}}(f)-\gamma_{T_{n}}(f)|+\sum_{u\in\mathcal{U}_{\epsilon_{n}}^{(1)}}X_{u}\frac{1}{n}\,.
\end{multline*}
\uline{Let \mbox{$\iota>0$}}. We fix $\omega$ in $\Omega$. Almost
surely, there exists $n_{0}$ such that, for $n\geq n_{0}$, $|\gamma_{T_{n+1}}(f)-\gamma_{T_{n}}(f)|<\iota$.
For $n\geq n_{0}$, we can then write (still with $f=\Id$):
\begin{equation}
\sum_{u\in\mathcal{U}_{\epsilon_{n}}^{(2)}}X_{u}\leq\frac{n+1}{n(1-e^{-a})}\left(\iota+\sum_{u\in\mathcal{U}_{\epsilon_{n}}^{(2)}}\sum_{v\in\boldsymbol{d}(u)}X_{v}\frac{1}{n+1}+\sum_{u\in\mathcal{U}_{\epsilon_{n}}^{(1)}}X_{u}\frac{1}{n}\right)\leq\frac{n+1}{n(1-e^{-a})}\left(\iota+\frac{1}{n}\right)\,.\label{eq:maj_U-2}
\end{equation}
Let $n\geq n_{0}$ and $t$ in $(T_{n},T_{n+1})$. We can decompose
\[
\mathcal{U}_{\epsilon_{n}}=\mathcal{U}_{\epsilon_{n}}^{(1)}(t)\sqcup\mathcal{U}_{\epsilon_{n}}^{(2)}(t)\text{ where }\mathcal{U}_{\epsilon_{n}}^{(1)}(t)=\mathcal{U}_{\epsilon_{n}}\cap\mathcal{U}_{e^{-t}}\,,\,\mathcal{U}_{\epsilon_{n}}^{(2)}(t)=\mathcal{U}_{\epsilon_{n}}\backslash\mathcal{U}_{e^{-t}}\,.
\]
For $u$ in $\mathcal{U}_{\epsilon_{n}}\backslash\mathcal{U}_{\epsilon_{n}}^{(1)}(t)$,
we set $\boldsymbol{d}(u,t)=\{v\in\mathcal{U}_{e^{-t}}\,:\,\boldsymbol{a}(v)=u\}$.
For any continuous $f$ from $[0,1]$ to $\mathcal{\R}$, there exists
$n_{1}\in\N^{*}$ such that, for all $x,y\in[0,1]$, $|x-y|<1/n_{1}\Rightarrow|f(x)-f(y)|<\iota$.
Suppose that $n\geq n_{0}\vee n_{1}$. Then we have (for all $t\in[T_{n},T_{n+1}]$),
\begin{multline}
|\gamma_{t}(f)-\gamma_{T_{n}}(f)|\leq\left|\sum_{u\in\mathcal{U}_{\epsilon_{n}}^{(1)}(t)}X_{u}f(e^{t}X_{u})-\sum_{u\in\mathcal{U}_{\epsilon_{n}}^{(1)}(t)}X_{u}f(nX_{u})\right|\\
+\left|\sum_{u\in\mathcal{U}_{\epsilon_{n}}^{(2)}(t)}X_{u}f(e^{t}X_{u})-\sum_{u\in\mathcal{U}_{\epsilon_{n}}^{(2)}(t)}X_{u}f(nX_{u})\right|\\
\leq\sum_{u\in\mathcal{U}_{\epsilon_{n}}^{(1)}(t)}X_{u}\iota+2\sum_{u\in\mathcal{U}_{\epsilon_{n}}^{(2)}(t)}X_{u}\Vert f\Vert_{\infty}\\
\text{(using Equation \ref{eq:maj_U-2})}\leq\iota+2\Vert f\Vert_{\infty}\frac{n+1}{n(1-e^{-a})}\left(\iota+\frac{1}{n}\right)\,.\label{eq:maj-temps-continu}
\end{multline}
Equations (\ref{eq:conv-discrete}) and (\ref{eq:maj-temps-continu})
prove the desired result.
\end{proof}
\begin{thm}
[Central-limit Theorem] \label{thm:central-limit}Let $q$ be in
$\N^{*}$. For functions $f_{1}$, \ldots{}, $f_{q}$ which are continuous
and in $\Bsym(1)$, we have
\[
\epsilon^{-q/2}(\gamma_{T}(f_{1}),\dots,\gamma_{T}(f_{q}))\underset{T\rightarrow+\infty}{\overset{\text{law}}{\longrightarrow}}\mathcal{N}(0,(K(f_{i},f_{j}))_{1\leq i,j\leq q})\,(\epsilon=e^{-T})
\]
($K$ is given in Equation (\ref{eq:def-K})).
\end{thm}

\begin{proof}
Let $f_{1}$, \ldots{}, $f_{q}$ $\Bsym(1)$ and $v_{1},\dots,v_{q}\in\R$.
\\
\uline{First}, we develop the product below 
\begin{multline*}
\prod_{u\in\mathcal{U}_{\epsilon}}\left(1+\sqrt{\epsilon}\frac{X_{u}}{\epsilon}(iv_{1}f_{1}+\dots+iv_{q}f_{q})\left(\frac{X_{u}}{\epsilon}\right)\right)=\\
\exp\left(\sum_{u\in\mathcal{U}_{\epsilon}}\log\left[1+\sqrt{\epsilon}\Id\times(iv_{1}f_{1}+\dots+iv_{q}f_{q})\left(\frac{X_{u}}{\epsilon}\right)\right]\right)=\\
\text{(for }\epsilon\text{ small enough)}\\
\exp\left(\sum_{u\in\mathcal{U_{\epsilon}}}\sum_{k\geq1}\frac{(-1)^{k+1}}{k}\epsilon^{k/2}(\Id\times(iv_{1}f_{1}+\dots+iv_{q}f_{q}))^{k}\left(\frac{X_{u}}{\epsilon}\right)\right)=\\
\text{(because, for \ensuremath{u\in\mathcal{U_{\epsilon}}}, \ensuremath{X_{u}/\epsilon\leq1} a.s.)}\\
\exp\left(\frac{1}{\sqrt{\epsilon}}\gamma_{T}(iv_{1}f_{1}+\dots+iv_{q}f_{q})+\frac{1}{2}\gamma_{T}(\Id\times(v_{1}f_{1}+\dots+v_{q}f_{q})^{2})+R_{\epsilon}\right)\,,
\end{multline*}
where 
\begin{eqnarray*}
R_{\epsilon} & = & \sum_{k\geq3}\sum_{u\in\mathcal{U}_{\epsilon}}\frac{(-1)^{k+1}}{k}\epsilon^{k/2-1}X_{u}\left(\frac{X_{u}}{\epsilon}\right)^{k-1}(iv_{1}f_{1}+\dots+iv_{q}f_{q})^{k}\left(\frac{X_{u}}{\epsilon}\right)\\
 & = & \sum_{k\geq3}\frac{(-1)^{k+1}}{k}\epsilon^{k/2-1}\gamma_{T}((\Id)^{k-1}(iv_{1}f_{1}+\dots+iv_{q}f_{q})^{k})\,,
\end{eqnarray*}
\[
|R_{\epsilon}|\leq\sum_{k\geq3}\frac{\epsilon^{k/2-1}}{k}(|v_{1}|\Vert f_{1}\Vert_{\infty}+\dots+|v_{q}|\Vert f_{q}\Vert_{\infty})=O(\sqrt{\epsilon})\,.
\]
We have, for some constant $C$, 
\begin{multline*}
\E\left(\left|\exp\left(\frac{1}{\sqrt{\epsilon}}\gamma_{T}(iv_{1}f_{1}+\dots+iv_{q}f_{q})+\frac{1}{2}\gamma_{T}(\Id\times(v_{1}f_{1}+\dots+v_{q}f_{q})^{2})+R_{\epsilon}\right)\right.\right.\\
\left.\left.-\exp\left(\frac{1}{\sqrt{\epsilon}}\gamma_{T}(iv_{1}f_{1}+\dots+iv_{q}f_{q})+\frac{1}{2}\eta(\Phi(\Id\times(v_{1}f_{1}+\dots+v_{q}f_{q})^{2})\right)\right|\right)\\
\leq\E\left(C\left|\frac{1}{2}\gamma_{T}(\Id\times(v_{1}f_{1}+\dots+v_{q}f_{q})^{2})-\frac{1}{2}\eta(\Phi(\Id\times(v_{1}f_{1}+\dots+v_{q}f_{q})^{2})+R_{\epsilon}\right|\right)\\
\text{(by Proposition \ref{prop:convergence-ps})}\underset{\epsilon\rightarrow0}{\longrightarrow}0\,.
\end{multline*}

\uline{Second}, we develop the same product in a different manner.
We have
\begin{multline*}
\prod_{u\in\mathcal{U}_{\epsilon}}\left(1+\sqrt{\epsilon}\frac{X_{u}}{\epsilon}(iv_{1}f_{1}+\dots+iv_{q}f_{q})\left(\frac{X_{u}}{\epsilon}\right)\right)=\\
\sum_{k\geq0}\epsilon^{-k/2}i^{k}\sum_{1\leq j_{1},\dots,j_{k}\leq q}v_{j_{1}}\dots v_{j_{k}}\sum_{\begin{array}{c}
u_{1},\dots,u_{k}\in\mathcal{U_{\epsilon}}\\
u_{1}<\dots<u_{k}
\end{array}}X_{u_{1}\dots}X_{u_{k}}f_{j_{1}}\left(\frac{X_{u_{1}}}{\epsilon}\right)\dots f_{j_{k}}\left(\frac{X_{u_{k}}}{\epsilon}\right)=\\
\sum_{k\geq0}\epsilon^{-k/2}i^{k}\sum_{1\leq j_{1},\dots,j_{k}\leq q}v_{j_{1}}\dots v_{j_{k}}\frac{1}{k!}\gamma_{T}^{\odot k}(f_{j_{1}}\otimes\dots\otimes f_{j_{k}})\,.
\end{multline*}
By Corollary \ref{cor:maj-fonction-centree}, we have, for all $k$,
\begin{multline*}
\left|\epsilon^{-k/2}\sum_{1\leq j_{1},\dots,j_{k}\leq q}v_{j_{1}}\dots v_{j_{k}}\frac{1}{k!}\E(\gamma_{T}^{\odot k}(f_{j_{1}}\otimes\dots\otimes f_{j_{k}}))\right|\\
\leq\frac{q^{k}\sup(|v_{1}|,\dots,|v_{q}|)^{k}\sup(\Vert f_{1}\Vert_{\infty},\dots,\Vert f_{q}\Vert_{\infty})^{k}}{k!}\left\{ K_{1}(q)+\Gamma_{1}^{q}C_{\text{tree}}(q)\left(\frac{1}{\delta}\right)^{q}q^{2}\right\} \,.
\end{multline*}
 So, by Corollary \ref{cor:lim-case-q-odd} and Proposition \ref{prop:conv-q-pair},
we get that 
\begin{multline*}
\E\left(\prod_{u\in\mathcal{U}_{\epsilon}}\left(1+\sqrt{\epsilon}\frac{X_{u}}{\epsilon}(iv_{1}f_{1}+\dots+iv_{q}f_{q})\left(\frac{X_{u}}{\epsilon}\right)\right)\right)\\
\underset{\epsilon\rightarrow0}{\longrightarrow}\sum_{\begin{array}{c}
k\geq0\\
k\text{ even}
\end{array}}(-1)^{k/2}\sum_{1\leq j_{1},\dots,j_{k}\leq q}\frac{1}{k!}\sum_{I\in I_{k}}\prod_{\{a,b\}\in I}V(v_{j_{a}}f_{j_{a}},v_{j_{b}}f_{j_{b}})\\
=\sum_{\begin{array}{c}
k\geq0\\
k\text{ even}
\end{array}}\frac{(-1)^{k/2}}{2^{k/2}(k/2)!}\sum_{1\leq j_{1},\dots,j_{k}\leq q}V(v_{j_{1}}f_{j_{1}},v_{j_{2}}f_{j_{2}})\dots V(f_{j_{k-1}},f_{j_{k}})\\
=\sum_{\begin{array}{c}
k\geq0\\
k\text{ even}
\end{array}}\frac{(-1)^{k/2}}{2^{k/2}(k/2)!}\left(\sum_{1\leq j_{1},j_{2}\leq q}v_{j_{1}}v_{j_{2}}V(f_{j_{1}},f_{j_{2}})\right)^{k/2}\\
=\exp\left(-\frac{1}{2}\sum_{1\leq j_{1},j_{2}\leq q}v_{j_{1}}v_{j_{2}}V(f_{j_{1}},f_{j_{2}})\right)\,.
\end{multline*}
\uline{In conclusion}, we have
\begin{multline*}
\E\left(\exp\left(\frac{1}{\sqrt{\epsilon}}\gamma_{T}(iv_{1}f_{1}+\dots+iv_{q}f_{q})\right)\right)\\
\underset{\epsilon\rightarrow0}{\longrightarrow}\exp\left(-\frac{1}{2}\eta(\Phi(\Id\times(v_{1}f_{1}+\dots+v_{q}f_{q})^{2}))-\frac{1}{2}\sum_{1\leq j_{1},j_{2}\leq q}v_{j_{1}}v_{j_{2}}V(f_{j_{1}},f_{j_{2}})\right)\,.
\end{multline*}
So we get the desired result with, for all $f$, $g$,
\begin{equation}
K(f,g)=\eta(\Phi(\Id\times fg)+V(f,g))\label{eq:def-K}
\end{equation}
($V$ is defined in Equation (\ref{eq:def-V})).
\end{proof}

\bibliographystyle{amsalpha}
\bibliography{biblio-fragmentation}

\end{document}